\newif\ifieeetran
\IfFileExists{IEEEtran.cls}{%
  \ieeetrantrue
  \documentclass[10pt,draftclsnofoot,onecolumn]{IEEEtran}
}{%
  \ieeetranfalse
  \documentclass[10pt]{article}
  \usepackage[margin=25mm]{geometry}
}
\usepackage{color}
\usepackage{xcolor}
\usepackage{tikz}
\usetikzlibrary{positioning,arrows.meta,calc,fit}
\usepackage{listings}
\usepackage{fancyhdr}
\usepackage{parskip}
\usepackage[most]{tcolorbox}
\usepackage[T1]{fontenc}
\usepackage[utf8]{inputenc}
\usepackage{amsmath,amssymb,amsthm,mathtools,bm}
\usepackage{booktabs,tabularx,array}
\usepackage{longtable}
\usepackage{enumitem}
\usepackage{setspace}
\usepackage{eurosym}
\usepackage{cite}
\usepackage{url}
\usepackage[hidelinks]{hyperref}
\usepackage[capitalise,noabbrev]{cleveref}
\ifieeetran\else
  \newenvironment{IEEEkeywords}{\par\noindent\textbf{Index Terms---}}{\par}
\fi

\allowdisplaybreaks[2]
\setlist{topsep=2pt,itemsep=0pt,parsep=0pt,partopsep=0pt}
\AtBeginDocument{%
  \setlength{\abovedisplayskip}{8pt plus 2pt minus 2pt}%
  \setlength{\belowdisplayskip}{8pt plus 2pt minus 2pt}%
  \setlength{\abovedisplayshortskip}{4pt plus 2pt minus 1pt}%
  \setlength{\belowdisplayshortskip}{6pt plus 2pt minus 1pt}%
  \setlength{\jot}{3pt}%
}

\usepackage[capitalise]{cleveref}
\crefname{algorithm}{algorithm}{algorithms}
\usepackage{url}

\usepackage{float}

\usepackage{algorithm}
\usepackage{algorithmicx}
\usepackage{algpseudocode}

\theoremstyle{plain}
\newtheorem{theorem}{Theorem}
\newtheorem{proposition}[theorem]{Proposition}

\newtheorem{corollary}[theorem]{Corollary}
\newtheorem{definition}[theorem]{Definition}
\newtheorem{assumption}[theorem]{Assumption}
\newtheorem{problem}[theorem]{Problem}
\newtheorem{remark}[theorem]{Remark}

\newcommand{\norm}[1]{\left\Vert #1\right\Vert}

\newcommand{\abs}[1]{\left|#1\right|}

\def\field#1{\mathbb #1}%
\def\R{\field{R}}%
\def\N{\field{N}}%
\def\Z{\field{Z}}%

\newcommand{\W}{\ensuremath{\mathcal W}}
\newcommand{\Rn}[1][n]{\R^{#1}}
\newcommand{\Rp}{\R_{\geq 0}}
\newcommand{\Rsp}{\R_{> 0}}
\newcommand{\Zp}{\Z_{\geq 0}}

\newcommand{\Nset}{\mathcal{N}}
\newcommand{\Eset}{\mathcal{E}}
\newcommand{\Iset}{\mathcal{I}}
\newcommand{\Bset}{\mathcal{B}}
\newcommand{\Rset}{\mathcal{R}}
\newcommand{\Gset}{\mathcal{G}}
\newcommand{\Tset}{\mathcal{T}}
\newcommand{\TsetE}{\mathcal{T}_{E}}
\newcommand{\Mset}{\mathcal{M}}
\newcommand{\Nbr}[1]{\mathcal{N}_{#1}}
\newcommand{\Xset}{\mathcal{X}}
\newcommand{\DeltaID}{\Delta_{\mathrm{ID}}}
\newcommand{\DeltaDA}{\Delta_{\mathrm{DA}}}
\newcommand{\TUTC}{\mathbb{T}_{\mathrm{UTC}}}
\newcommand{\Tr}{^{\!\top}}
\newcommand{\defeq}{\vcentcolon=}
\newcommand{\pconv}{p^{\mathrm{conv}}}
\newcommand{\pgen}{p^{\mathrm{g}}}
\newcommand{\ppump}{p^{\mathrm{pump}}}
\newcommand{\pplant}{p^{\mathrm{pl}}}
\newcommand{\pc}{p^{\mathrm{c}}}
\newcommand{\pd}{p^{\mathrm{d}}}
\newcommand{\pres}{p^{\mathrm{res}}}
\newcommand{\preg}{p^{\mathrm{reg}}}
\newcommand{\pflow}{p^{\mathrm{flow}}}
\newcommand{\Pmot}{P^{\mathrm{mot}}}
\newcommand{\Pav}{P^{\mathrm{av}}}
\newcommand{\Pres}{P^{\mathrm{res,max}}}
\newcommand{\speak}{s^{\mathrm{pk}}}
\newcommand{\etareg}{\eta^{\mathrm{reg}}}

\title{Distributed Intraday Energy Management for 16.7-Hz Railway Power Systems---Part I:\\
Theoretical Foundations}

\author{Navid Noroozi%
\thanks{N. Noroozi is with SIGNON Deutschland GmbH (DB InfraGO AG), Europaplatz 2, 10557 Berlin, Germany, \texttt{navid.n.noroozi@deutschebahn.com}.}}

\begin{document}
\maketitle

\begin{abstract}
\begin{singlespace}
In a series of three papers, we propose a distributed intraday scenario planner for energy management in 16.7-Hz single-phase railway power systems, which is realized and tested as open-source Python package \texttt{bahnstrom\_ems v0.9.1}~\cite{noroozi2026bahnstromems}.
The first part of our work develops the mathematical foundation of a two-layer energy management system for a 16.7-Hz single-phase railway power network, where the network is an interconnection of several control areas.
Given a day-ahead plan, last measurement of stored energy in wayside batteries, and  forecasts of the tractive, available regenerative and renewable energy, we provide a mathematical modeling of the 16.7-Hz single-phase railway power network in intraday time-scale.
Then the intraday energy planning is formulated as a risk-neutral convex stochastic optimization problem.
The overall intraday planning is shown to be a sparse convex quadratic program.
Hence, this problem is solved recursively by adopting a distributed scenario model predictive control (MPC) setting.
The corresponding area problems, coordination updates, convergence conditions, and closed-loop key-performance-index accumulation are derived.
In practice, detailed train-motion and microscopic motoring and regenerative-power profiles are usually needed for accurate simulations and numerical validation purposes.
Due to data privacy and security concerns, detailed train-motion and microscopic motoring and regenerative-power profiles in a real power network may not be publicly available.
Therefore, we also propose a mathematically-solid approach to estimate such microscopic information by converting causally timetable and train-motion information into quarter-hour motoring and regenerative-power profiles while exactly preserving each hourly energy target.
\end{singlespace}
\end{abstract}

\begin{IEEEkeywords}
Railway power systems, energy management, regenerative braking, model predictive control, scenario optimization, distributed convex optimization.
\end{IEEEkeywords}

\section{Introduction}
\label{sec:introduction}

The large railway power systems in middle Europe, Germany, Austria, and Switzerland, as well as in north Europe, Norway and Sweden, operate at the nominal frequency of 16.7~Hz and use single-phase traction transmission and supply equipment that is electrically and operationally distinct from the surrounding three-phase 50-Hz public grid.
Within each connected railway network, converter plants, traction substations, railway-owned generating stations, direct renewable sources, and wayside battery energy storage systems are geographically distributed.
Train acceleration and braking move large amounts of power between these assets on a timetable-driven time-scale.
This turns the railway energy supply-demand problem into a constrained, multi-period dispatch problem, where both temporal and spatial dimensions do matter.

The railway system intrinsically operates in multi time-scale.
Time-table, unit commitment and market positions must be prepared before delivery, whereas deviations in train operation, regenerative-braking availability, renewable production, and asset availability must be handled during the day.
A single centralized formulation can represent both time scales, but it obscures their different information patterns and operating requirements.
The railway sector has been facing rapidly increasing disruptions in the train operation which may challenge optimal energy dispatching in a shorter time-scale (e.g., 15 minute intraday, 5 minute real-time operation).
Nevertheless, it is necessary to get trains (and other infrastructural assets of the railway sector) operated energy-efficiently. 
This requires to develop innovative decision making mechanisms which can provide solutions for an optimal energy and/or train operation in a shorter amount of time and, hence, they are more agile to act against operational disruptions.
In this work, we break the complexity of energy management in a 16.7-Hz power network into the scale of smaller areas (clusters) of the overall 16.7-Hz power network.
For each area, a distributed intraday controller uses fresh measurements and forecasts to coordinate the faster corrective dispatch over the 16.7-Hz network.

Regenerative braking is central to this formulation.
Reviews of railway energy-saving measures identify receptive networks, reversible substations, and wayside storage as major means of recovering braking  nergy~\cite{gonzalezgil2013,ratniyomchai2014}.
If traction demand is represented only by a net metered load, however, motoring and braking have already been combined and the amount of braking power actually accepted by the network is not a
control decision.
We instead distinguish gross motoring power, available regenerative power, and accepted regenerative power.
This makes the recovery ratio a traceable output of the dispatch rather than a prescribed parameter.

Centralized--decentralized railway energy-management architectures have previously been developed for railway smart grids~\cite{khayyam2015}.
In particular, regenerative-energy management and wayside storage are well established in the railway literature~\cite{gonzalezgil2013,ratniyomchai2014}.
Multi-time-scale traction-power dispatch by model predictive control (MPC) has also been investigated~\cite{chen2020}.
In the broader control literature, scenario MPC, sample-average objectives, and multi-stage stochastic programming are established methods~\cite{bernardini2009,bernardini2012,mesbah2016}.
Then alternating direction method of multipliers (ADMM) is a standard decomposition method for separable convex programs with linear coupling~\cite{boyd2011}, which is widely used in the context of distributed optimization--based control for energy systems~\cite{Khaki.2019,Braun.2018}.
The linear active power network approximation and convex economic-dispatch models used below likewise follow standard steady-state power-system practice~\cite{zimmerman2011}.

Our contributions are basically divided into to main categories: 1) theoretical foundations for generic modeling and energy management in intraday time-scale of 16.7-Hz power systems; 2) modular, open-source software realization of our theoretical results, which is, in turn, validated through a semi-real openly available dataset of Swiss traction network~\cite{theiler2026}.
For the sake of brevity, we break the presentation of our results into a three-part study, where Part~I turns the theoretical concepts into modular, multi-layered open-source software, and finally Part~III evaluates and validates the theoretical results and the companion software through numerical simulations.
In particular, Part~I first provides a generic modeling of an intraday 16.7~Hz railway power system that includes a power converter connected to the public grid, railway-owned power plants, wayside storage, renewable injection, regenerative-braking recovery, and power exchange between control areas.
Given a day-ahead plan, last measurement of stored energy in wayside batteries, and  forecasts of the tractive, available regenerative and renewable energy, we provide a mathematical modeling of the 16.7-Hz single-phase railway power network in intraday time-scale.
Then the intraday energy planning is formulated as a risk-neutral convex optimization problem.
Using the Sample Average Approximation (SAA)~\cite{kleywegt2002}, the the risk-neutral expected-cost problem is approximated by the average cost over a finite ensemble of sampled uncertainty trajectories.\footnote{In Part~I, we only assume that the sampled scenarios are given. The forecasting layer is presented in details in Part~II.}
By showing the sparsity of the overall optimization problem, we efficiently solve the resulting optimization problem using consensus-ADMM.
Running this setting recursively (in 15-minute resolution) leads to a distributed scenario model predictive control (MPC) setting.

Last but not least, we note that the intraday planner requires separate gross motoring and available regenerative powers on the 15-minute control grid, whereas many public railway datasets provide only hourly energy or mean power and incomplete train-level timing.
Hourly averaging can place acceleration and braking in the same aggregate value and thereby hide the intervals in which braking power exceeds simultaneous motoring demand.
Arbitrary interpolation would restore a finer grid but would not restore railway timing.
Therefore, we also propose a mathematically-solid approach to estimate such microscopic information by converting causally timetable and train-motion information into quarter-hour motoring and regenerative-power profiles while exactly preserving each hourly energy target.

\section{Preliminaries}
\label{sec:two-layer}

This section introduces notation, symbols and the structure of railway power system considered in the current work and the clocks used in the multi-time scale system.
The first subsection reviews the notation.
The following subsections briefly specify the day-ahead interface and the receding-horizon operation of the intraday layer.
We note that a stage number, a scenario index, time indexes denoting multiple time-scales, and the absolute UTC time denoting the real operational time of the system shall \emph{not} be used interchangeably.

\subsection{Notation}

In this paper, $\Rp (\Rsp)$ and $\Zp (\N)$ denote the nonnegative
(positive) real numbers and the nonnegative (positive) integers,
respectively.
Given a set $\mathcal{S} \subseteq \Rn$, $\mathcal{S}^\ell := \underbrace{\mathcal{S} \times \dots \times \mathcal{S}}_{\ell \, \, \mathrm{times}}$.
The $i$th component or partition of $v \in \Rn$ is denoted by $v_i$.
For any $v \in \Rn$, $v^\top$ denotes its transpose.
We write $(v,w)$ to represent $[v^\top,y^\top]^\top$ for $x \in \Rn,y \in \R^p$.
In a more general way, the vector operator $\operatorname{col}_{i=1}^{\ell}(v_i)$ vertically stacks the vectors $v_1\in\R^{n_1},\dots,v_\ell\in\R^{n_1}$ into one long column vector.
Let $d_1,d_2 \in \N$, we denote a zero matrix (resp. vector) $\mathbf{0}_{d_1\times d_2}$ (resp. $\mathbf{0}_{d_1}$).
We may omit the superscripts if the dimension is clear from the context.

For $x \in \Rn$, we, respectively, denote the Euclidean norm and the maximum norm by $\norm{x}$ and by $\abs{x}_\infty$.
For a set $\Omega_0$, a sigma-algebra $\mathcal F$ is a collection of subsets of $\Omega_0$, called events, such that $\Omega_0\in\mathcal F$ and $\mathcal F$ is closed under complements (relative to $\Omega_0$) and countable unions, i.e. $\Omega_0\setminus \mathcal A\in\mathcal F$ for every $\mathcal A\in\mathcal F$, and $\bigcup_{k=1}^{\infty}\mathcal A_k\in\mathcal F$ for every sequence of events $\mathcal A_k\in\mathcal F$.
A probability space $(\Omega_0,\mathcal F,\mathbb P_0)$ consists of a set $\Omega_0$ of
elementary outcomes, a sigma-algebra $\mathcal F$ of events, and a probability measure
$\mathbb P_0:\mathcal F\to[0,1]$.  The sigma-algebra contains $\Omega_0$ and is closed
under complements and countable unions.  The measure satisfies $\mathbb P_0(\Omega_0)=1$
and $\mathbb P_0(\bigcup_{k=1}^{\infty} A_k)=\sum_{k=1}^{\infty}\mathbb P_0(A_k)$
for pairwise disjoint events $A_k\in\mathcal F$.  

\subsection{Railway power system topology, and clocks}

Represent a railway power network by the graph $\mathcal G=(\Nset,\Eset)$, where $\Nset=\{1,\ldots,N\}$ is the set of electrical control areas and $\Eset\subset\Nset\times\Nset$ is the set of railway-power transmission lines stored with a fixed reference orientation.
By convention, we set the time-resolution of the intraday energy planner to 15 minutes.
We then pick a real-valued initial time no later than the beginning of the study and identify every UTC time stamp with its elapsed time from that initial.
The wall-clock time axis is then represented by $\TUTC\subset\R_{\ge0}$.
The intraday sampling \emph{duration} is $\DeltaID=15\ \mathrm{min}=\tfrac14\ \mathrm{h}$.
It is a positive real number with units of time.
Let $n\in\Zp$ be the closed-loop control-step counter and let $\tau_0\in\TUTC$ be the first wall-clock control time.  The $n$th \emph{control instant} is the absolute time $\tau_n\defeq\tau_0+n\DeltaID\in\TUTC$.
Thus $\tau_n$ is the UTC time at which the controller receives the current state, forms a new finite-horizon problem, and selects the next applied control.

For a problem formed at $\tau_n$, let $H\in\mathbb Z_{\ge1}$ and define the prediction-stage
index set $\Tset\defeq\{0,\ldots,H-1\}\subset\Zp$.
The symbol $t\in\Tset$ is exclusively a relative prediction-stage index.
A \emph{stage} is, in plain language, one 15-minute interval of the finite-horizon dispatch together with the controls, states, and exogenous quantities assigned to that interval, where the new information/realization becomes available.
More precisely, stage $t$ of the problem formed at $\tau_n$ begins at $\tau_{n,t}\defeq\tau_n+t\DeltaID$ and represents the half-open wall-clock interval $\mathcal I_{n,t}\defeq[\tau_{n,t},\tau_{n,t}+\DeltaID)$.
Hence $t=0$ is the current 15-minute interval and $t>0$ denotes a future interval relative to the same issue time.

The day-ahead plan has its own absolute start time $\tau_0^{\mathrm{DA}}\in\TUTC$, hourly duration $\DeltaDA=1\ \mathrm{h}$, and horizon length $H_{\mathrm{DA}}\in\mathbb Z_{\ge1}$.
Its integer index set is $\mathcal H^{\mathrm{DA}}=\{0,\ldots,H_{\mathrm{DA}}-1\}$.
The symbol $h\in\mathcal H^{\mathrm{DA}}$ is a generic coordinate of that fixed day-ahead grid.
It is neither an absolute time nor an intraday stage.  The day-ahead hour containing intraday stage $t$ of the problem formed at $\tau_n$ is the evaluated map $h_n(t)\defeq \left\lfloor\frac{\tau_{n,t}-\tau_0^{\mathrm{DA}}}{\DeltaDA}\right\rfloor$, $t\in\Tset$, provided $h_n(t)\in\mathcal H^{\mathrm{DA}}$.
Thus $h$ denotes an arbitrary day-ahead hour, whereas $h_n(t)$ denotes the particular day-ahead hour selected by the absolute time of one intraday stage.
The two symbols are not interchangeable.

Within the intraday planner, an \emph{issue time} is an absolute UTC control instant
$\tau_n\in\TUTC$ at which a measurement set, schedule, forecast, or availability envelope
is declared available for operational use.
At that instant, measurements are acquired, forecasts and availability envelopes are issued,
and the optimization data are frozen using only information available no later than
$\tau_n$.  For an object $y$ issued at $\tau_n$, causality requires $y$ to be
$\mathcal F_{\tau_n}$-measurable, where $\mathcal F_{\tau_n}$ contains all measurements,
confirmed schedules, issued forecasts, active day-ahead decisions, and physical states
available by that time.  An admissible applied decision $u_n$ must likewise be
$\mathcal F_{\tau_n}$-measurable.

At issue time $\tau_n$, the forecasting process supplies a finite set of joint future trajectories
$\omega_n^{(m)}$, $m\in\Mset:=\{1,\ldots,M\}$, with probabilities
$\pi_m>0$, $\sum_m\pi_m=1$.
All scenario data at stage $t=0$ equal the same current measurement. 
Only stages $t\ge1$ may branch.
This convention distinguishes the known present from uncertain future quantities and is used throughout the paper.

\subsubsection{Nomenclature} 
\begin{longtable}{@{}p{4.0cm}p{11.5cm}@{}}
\toprule
\textbf{Symbol} & \textbf{Meaning} \\
\midrule \endhead
$\Pmot_i$ & gross motoring power [MW] \\
$\Pav_i$ & \emph{available} regenerative braking power [MW] \\
$\Pres_i$ & available renewable power [MW] \\
$c^{E}_i,\,c^{E,\mathrm{x}}_i$ & import price / export credit [\euro/MWh] \\
\midrule
$\pconv_i$ & converter power drawn from the $50$\,Hz grid (import $>0$) \\
$q^{+}_i,q^{-}_i\ge0$ & positive/negative parts of $\pconv_i$ \\
$\pgen_g,\ppump_g\ge0$ & railway-plant generation and pumping power [MW]; $\ppump_g\equiv0$ for a unit without pumping capability \\
$\pplant_i$ & net railway-plant injection at area $i$: generation minus pumping [MW] \\
$\pc_i,\pd_i\ge0$, $E_i$ & BESS charge, discharge, state of charge [MWh] \\
$\pres_i$ & dispatched renewable power \\
$\preg_i$ & regen power actually re-injected \\
$\pflow_{e,i}$ & area $i$'s \emph{local copy} of the flow on incident edge $e$ \\
$\theta_{a,i}$ & area $i$'s \emph{local copy} of the voltage angle of area $a\in\{i\}\cup\Nbr{i}$ \\
$\speak_i\ge0$ & peak-above-target slack \\
$\ell_e$ & epigraph variable for the loss on edge $e$ \\
\bottomrule
\end{longtable}

\subsection{Day-ahead interface}

The day-ahead planner is represented only through its output.  For each applicable area and
hour it supplies
\begin{equation}
\mathcal D^{\mathrm{DA}}=
\left\{
\delta_i^{\mathrm{DA}}(h),\ b_i^{\mathrm{DA}}(h),\
E_i^{\mathrm{DA}}(h),\ \bar p_i^{\mathrm{DA}},\
p_g^{\mathrm{g,DA}}(h),\ p_g^{\mathrm{pump,DA}}(h)
\right\},
\label{eq:da-interface}
\end{equation}
where $\delta_i^{\mathrm{DA}}\in\{0,1\}$ is a frozen converter commitment,
$b_i^{\mathrm{DA}}$ is the cleared hourly mean converter exchange (a power),
$E_i^{\mathrm{DA}}$ is a battery-energy set-point, and
$\bar p_i^{\mathrm{DA}}$ is an import-peak target.  The optional schedules
$p_g^{\mathrm{g,DA}}$ and $p_g^{\mathrm{pump,DA}}$ provide reference generation and pumping
positions for railway-owned plant $g$. 
Unlike $\delta_i^{\mathrm{DA}}$, they are continuous
reference values, not commitment variables.  In the present formulation they are retained
in the audit and initialization record but are neither hard constraints nor tracking terms:
the plant remains redispatchable inside its latest issue-time envelope.  A secured planner
may additionally
supply a feasible 15-minute witness for audit and emergency use, but that witness is not a
decision variable of the intraday problem.

\begin{remark}
No day-ahead objective, binary logic, or mixed-integer constraint is required in this paper.
Any day-ahead scheduling method may be used if it provides the quantities in
\eqref{eq:da-interface} on an absolute-time grid and if the meanings and units of those
quantities agree with the intraday model.      
\end{remark}

\subsection{Intraday scenario-MPC layer}
\label{sec:scenario-preliminaries-paper}

Here we briefly review a a standard scenario MPC setting.
Given a probability space $(\Omega_0,\mathcal F,\mathbb P_0)$, let the sub-sigma-algebra $\mathcal F_{\tau_n}\subseteq\mathcal F$ represent all information available at control instant $\tau_n$.
Throughout the paper, a quantity is called $\mathcal F_{\tau_n}$-measurable if it is determined by the information available at $\tau_n$. In particular, an admissible applied control $u_n$ may depend on measurements, forecasts, schedules, and states available by $\tau_n$, but not on future realizations.

Fix a control instant $\tau_n$ and the information available at that instant. Let
$
w=\operatorname{col}{t=0}^{H-1} w(t),\
w(t)=\operatorname{col}{i\in\Nset}
\bigl(\Pmot_i(t),\Pav_i(t),\Pres_i(t)\bigr)
$
denote the random $H$-stage exogenous input trajectory, taking values in
$\W=(\R_{\ge0}^{3N})^H$.
Thus, one realization $\omega\in\W$ specifies the gross motoring demand, regenerative
availability, and renewable availability of all areas over the complete prediction horizon.
Let $\mathbb P_n$ denote the conditional probability distribution of $w$ given the information
available at $\tau_n$.

Let $x$ denote an real-valued optimization decision and let $J(x,w)$ denote its horizon cost under the disturbance $w$.
The corresponding risk-neutral stochastic optimization problem is
\begin{equation}
\inf_{x\in\Xset^{\mathrm{ad}}(\tau_n)}
\mathcal J_{\mathbb P_n}(x),
\qquad
\mathcal J_{\mathbb P_n}(x)
\defeq
\mathbb E_0[J(x,w)\mid\mathcal F_{\tau_n}]
=
\int_{\W} J(x,\omega),\mathbb P_n(\mathrm d\omega),
\label{eq}
\end{equation}
where $\Xset^{\mathrm{ad}}(\tau_n)$ contains the decisions satisfying the required causality
and system constraints. The problem is risk-neutral because its objective minimizes the
expected cost, without an additional variance or tail-risk penalty.

A \emph{scenario} $\omega^{(m)}\in\W$ is one realization of the complete $H$-stage exogenous input trajectory.
For $m\in\Mset=\{1,\ldots,M\}$ and $\pi_m>0$, $\sum_m\pi_m=1$, define
\begin{equation}
\widehat{\mathcal J}_M(x) \defeq \sum_{m\in\Mset}\pi_mJ(x,\omega^{(m)}).
\label{eq:saa-objective-prelim-paper}
\end{equation}
For i.i.d. trajectory sampling from $\mathbb P$, conditional on the issue-time data before
sampling, we have $\pi_m=1/M$.
This gives the \emph{sample average approximation} (SAA),
$\widehat{\mathcal J}_M(x)=M^{-1}\sum_{m=1}^M J(x,\omega^{(m)})$.

A prediction stage $t\in\Tset$ represents the interval
$\mathcal I_{n,t}=[\tau_n+t\DeltaID,\tau_n+(t+1)\DeltaID)$.
Scenario $m$ specifies inputs $\omega(t,m)$ over these same intervals, with
\begin{equation}
\omega^{(m)}=\operatorname{col}_{t=0}^{H-1}\omega(t,m),
\qquad
\omega(t,m)=\operatorname{col}_{i\in\Nset}
\bigl(\Pmot_i(t,m),\Pav_i(t,m),\Pres_i(t,m)\bigr).
\label{eq:scenario-stage-input-paper}
\end{equation}
Given the sampled exogenous inputs $\omega$, let $x$ collect a decision vector at stage $t$ along scenario $m$ and $\ell$ denote a stage cost.
Then we have
\begin{equation}
\widehat{\mathcal J}_M(x)=\sum_{m\in\Mset}\pi_m
\underbrace{\left[\sum_{t\in\Tset}\ell\bigl(t,x(t,m),\omega(t,m)\bigr)\right]}_{J(x,\omega^{(m)})}.
\label{eq:scenario-mpc-cost-prelim-paper}
\end{equation}
Section~\ref{sec:intraday-final} specifies the decisions vector $x$ and the stage cost $\ell$.
Let $\Xset_M^{\mathrm{fan}}(\tau_n)$ contain all possible finite scenario decision vectors.
The intraday MPC planner solves $\min_{x\in\Xset_M^{\mathrm{fan}}(\tau_n)}\widehat{\mathcal J}_M(x)$
at each control instant and applies the first component of the optimal control sequence.
At $\tau_{n+1}$, new measurements and forecasts are obtained and the problem is solved again~\cite{bernardini2009,mesbah2016}.

\section{Mathematical Model of the Intraday Dispatch}
\label{sec:intraday-model}

This section translates the two-layer architecture into the physical constraints of one
finite-horizon intraday problem.  
In particular, it develops network, storage, converter-plant,
railway-plant, renewable, market-coupling, and non-anticipativity constraints.  The aim is to
state exactly which quantities are measured or issued parameters and which quantities remain
intraday control decisions.


\subsection{Sets, orientation, and scenario data}

The subsets $\Iset,\Bset,\Rset\subseteq\Nset$ identify areas with a converter plant,
a wayside battery, and direct renewable injection, respectively.  Railway-owned generating
units form a separate set $\Gset$, with placement map $a:\Gset\to\Nset$ and
$\Gset_i\defeq\{g\in\Gset:a(g)=i\}$.  The distinction is physical: a generating unit is an
asset, whereas a control area is an electrical node.
For edge $e$ incident on area $i$, let $\sigma_{i,e}=+1$ if $i$ is the stored from-node and
$\sigma_{i,e}=-1$ if it is the to-node.  The degree of area $i$ is
$K_i=|\Nbr{i}|$.  For each $(i,t,m) \in \Nset \times \Tset \times \Mset$, the scenario data are the nonnegative gross motoring
power $\Pmot_i(t,m)$, available regenerative power $\Pav_i(t,m)$, and available direct
renewable power $\Pres_i(t,m)$.  Prices, asset limits, day-ahead quantities, and the
initial physical state are fixed data at issue time $\tau_n$.
The outer closed-loop index is suppressed inside one fixed problem to avoid burdening every
horizon variable with a third time index.  Thus, for example,
$\Pav_i(t,m)$ in this and the following sections means the data
$\Pav_{i,n}(t,m)$ of the problem issued at $\tau_n$. 

At stage zero, the train-power inputs are the causal present values available at
$\tau_n$ and are copied identically across the scenario fan:
$\Pmot_i(0,m)=\widehat P_{i,n}^{\mathrm{mot}}$, $\Pav_i(0,m)=\widehat P_{i,n}^{\mathrm{av}}$, $m\in\Mset$.
The hats denote current values supplied to the controller, not decision variables.  In an
operational installation they may be obtained from synchronized traction measurements and a
causal motoring/braking disaggregation or from a state estimate supported by the current
timetable.  
Future values
$\Pmot_i(t,m),\Pav_i(t,m)$ for $t\ge1$ are forecast-scenario inputs.  Thus the quantity
$\Pav_i(0,m)$ in \eqref{eq:regen-bounds-paper} is a measured or causally estimated input to
the intraday planner and is deliberately scenario-independent.

The double-index angle $\theta_{a,i}$ denotes area $i$'s local copy of the voltage angle
owned by area $a\in\{i\}\cup\Nbr{i}$.  Thus $\theta_{i,i}$ is area $i$'s own angle and
$\theta_{j,i}$ is area $i$'s copy of its neighbor $j$'s angle.  This convention is used
throughout the decomposition.

\subsection{Traceable regenerative-braking decision}

The accepted regenerative power is a decision variable satisfying
\begin{equation}
0\le\preg_i(t,m)\le\Pav_i(t,m).
\label{eq:regen-bounds-paper}
\end{equation}
The corresponding net traction load is
\begin{equation}
\xi_i(t,m)=\Pmot_i(t,m)-\preg_i(t,m).
\label{eq:net-traction-paper}
\end{equation}
The difference $\Pav_i-\preg_i$ is braking energy that cannot be accepted by the modeled
network and assets.  This explicit separation is the basis of the recovery KPI in
\cref{sec:kpi}.

\subsection{Nodal balance and linearized single-phase power flow}

For every $i\in\Nset$, $t\in\Tset$, and $m\in\Mset$, the active power balance is
\begin{equation}
\begin{aligned}
&\pconv_i(t,m)+\pplant_i(t,m)+\pd_i(t,m)-\pc_i(t,m)
+\pres_i(t,m)+\preg_i(t,m)\\
&\hspace{22mm}
-\sum_{e\in\Eset:\,i\in e}\sigma_{i,e}\pflow_{e,i}(t,m)
=\Pmot_i(t,m).
\end{aligned}
\label{eq:nodal-balance-paper}
\end{equation}
The sign convention gives a negative contribution for a positive flow leaving the from-node and a positive contribution at the to-node.

At the tertiary dispatch time scale, the single-phase railway network is represented by the
standard lossless linear active power relation adapted to local angle copies
\cite{zimmerman2011}:
\begin{align}
\pflow_{e,i}(t,m)
&=\sigma_{i,e}B_e\bigl(\theta_{i,i}(t,m)-\theta_{j,i}(t,m)\bigr),
\quad j\in\Nbr{i},
\label{eq:local-flow-paper}\\
|\pflow_{e,i}(t,m)|&\le\bar P_e,
\label{eq:line-limit-paper}
\end{align}
where $B_e>0$ is the susceptance of transmission line $e$ and $\bar P_e>0$ is its active power limit.
A reference-area consensus angle is fixed to zero later.  The approximation is an
active power scheduling model, not an electromagnetic or unbalanced AC-network model.

\subsection{Ohmic losses}
The exact ohmic loss on edge $e$ is the convex quadratic $R_e V_{\mathrm{nom}}^{-2}(\pflow_e)^2$, with resistance $R_e$ and nominal line voltage $V_{\mathrm{nom}}$.
To price approximate ohmic losses without introducing a nonlinear constraint, assign each edge once, to its stored from-node, and introduce an epigraph variable $\ell_e(t,m)$ with
\begin{equation}
\ell_e(t,m)\ge a_{e,\kappa}\pflow_{e,\mathrm{from}(e)}(t,m)+b_{e,\kappa},
\qquad \kappa=1,\ldots,K_\ell,
\label{eq:loss-cuts-paper}
\end{equation}
where the affine functions with constant $a_{e,\kappa},b_{e,\kappa}\in\R$ are tangents to
$R_eV_{\mathrm{nom}}^{-2}p^2$ at selected flow values.
Their pointwise maximum is a convex piecewise-linear under-estimator.

\subsection{Wayside battery}

For $i\in\Bset$, define the energy-state index set
$\TsetE\defeq\{0,\ldots,H\}$ and let $E_i(t,m)$ be the stored energy at the start of stage
$t\in\Tset$, with $E_i(H,m)$ denoting the energy at the end of the horizon.  The measured
initial value is $E_i(0,m)=E_i^0$.  The standard discrete energy balance used in railway
storage scheduling~\cite{ratniyomchai2014} is
\begin{align}
E_i(t+1,m)
&=E_i(t,m)+\DeltaID\left(\eta_i^{\mathrm c}\pc_i(t,m)
-\frac{1}{\eta_i^{\mathrm d}}\pd_i(t,m)\right),
&&t\in\Tset,
\label{eq:bess-dynamics-paper}\\
\underline E_i\le E_i(s,m)&\le\overline E_i,
&&s\in\TsetE,
\nonumber\\
0\le\pc_i(t,m),\pd_i(t,m)&\le\overline P_i^{\mathrm b},
&&t\in\Tset,
\label{eq:bess-box-paper}\\
E_i(H,m)&\ge E_i^{\mathrm{term}}.
\label{eq:bess-terminal-paper}
\end{align}
Here $\eta_i^{\mathrm c},\eta_i^{\mathrm d}\in(0,1]$.  Separate nonnegative charging and
discharging variables keep the model convex.  The formulation does not impose the
nonconvex complementarity $\pc_i\pd_i=0$.

\begin{remark}[Charging/discharging relaxation]
A positive throughput cost discourages simultaneous charge and discharge but, for
$\eta_i^{\mathrm c}\eta_i^{\mathrm d}<1$, does not by itself prove complementarity for every
possible price and network condition.  Exact exclusion would require a mode variable or a
complementarity constraint and would destroy the convex-QP structure.
In the lossless special case $\eta_i^{\mathrm c}=\eta_i^{\mathrm d}=1$, reducing equal positive charge and discharge by
their minimum preserves both balance and stored energy and strictly lowers any positive throughput cost.
Hence an optimum with simultaneous operation cannot then exist.
\end{remark}

\subsection{Converter plant}

For $i\in\Iset$, positive and negative exchange are separated as
\begin{align}
\pconv_i(t,m)&=q_i^+(t,m)-q_i^-(t,m),
\qquad q_i^+(t,m),q_i^-(t,m)\ge0,
\label{eq:converter-split-paper}\\
\delta_i^{\mathrm{DA}}(h_n(t))\underline P_i^{\mathrm{cv}}
&\le\pconv_i(t,m)\le
\delta_i^{\mathrm{DA}}(h_n(t))\overline P_i^{\mathrm{cv}},
\label{eq:frozen-commitment-paper}\\
|\pconv_i(0,m)-p_{i,n}^{\mathrm{conv,prev}}|
&\le\Delta\overline P_i^{\mathrm{cv}},
\label{eq:converter-initial-ramp-paper}\\
|\pconv_i(t,m)-\pconv_i(t-1,m)|&\le\Delta\overline P_i^{\mathrm{cv}},
\label{eq:converter-ramp-paper}
\end{align}
where $t\in\Tset\setminus\{0\}$, $p_{i,n}^{\mathrm{conv,prev}}$ is the exchange measured immediately before
$\tau_n$.
A reversible static converter has $\underline P_i^{\mathrm{cv}}<0$.
An import-only converter plant has $\underline P_i^{\mathrm{cv}}=0$.
In this tertiary active power model, the latter case also covers an aggregated rotary converter plant that can import from the
public grid but is not credited with controlled reverse power.  This abstraction does not
represent the electromechanical dynamics, reactive power behavior, losses, or internal
unit-level controls of a rotary converter.  Since the converter-plant commitment
$\delta_i^{\mathrm{DA}}(h_n(t))$ is fixed before the intraday problem is formed,
\eqref{eq:frozen-commitment-paper} is an affine bound rather than an online binary
constraint.

\begin{proposition}
\label{prop:no-sim-converter}
At a stage for which the import price exceeds the export credit,
$c_i^E(t)>c_i^{E,\mathrm x}(t)$, every optimum satisfies
$q_i^+(t,m)q_i^-(t,m)=0$.
\end{proposition}

\begin{proof}
Assume $q_i^+>0$ and $q_i^->0$ and set $\delta=\min\{q_i^+,q_i^-\}>0$.  Replacing both
variables by $q_i^+-\delta$ and $q_i^--\delta$ preserves their difference
$\pconv_i$ and therefore preserves every constraint.  The stage cost changes by
$-\DeltaID(c_i^E-c_i^{E,\mathrm x})\delta<0$, contradicting optimality.
This completes the proof.
\end{proof}

\subsection{Railway-owned generation and pumping}

For $g\in\Gset$, let $\pgen_g(t,m)\ge0$ and $\ppump_g(t,m)\ge0$ denote generation and
pumping.  Their exogenous, issue-time admissible envelopes are
$\overline P_g^{\mathrm g}(t,m)$ and $\overline P_g^{\mathrm p}(t,m)$:
\begin{align}
p_g^{\mathrm{pl}}(t,m)&=\pgen_g(t,m)-\ppump_g(t,m),
\qquad
\pplant_i(t,m)=\sum_{g\in\Gset_i}p_g^{\mathrm{pl}}(t,m),
\label{eq:plant-net-paper}\\
0\le\pgen_g(t,m)&\le\overline P_g^{\mathrm g}(t,m),
\qquad
0\le\ppump_g(t,m)\le\overline P_g^{\mathrm p}(t,m).
\label{eq:plant-bounds-paper}
\end{align}
At issue time $\tau_n$, the values
$\overline P_g^{\mathrm g}(t,m)$ and $\overline P_g^{\mathrm p}(t,m)$ are supplied by a
causal plant-availability interface for the absolute interval $\mathcal I_{n,t}$.  They are
upper bounds, not optimized plant powers and not necessarily meter readings. 
Pumping is treated as a contemporaneous controllable load here.
This is the standard convex active power injection abstraction used in steady-state economic dispatch~\cite{zimmerman2011}.
An explicit hydrothermal reservoir scheduling~\cite{heredia1995} is out of scope of the current study and left as our future work.

\begin{remark}
The absence of a binary commitment variable in \eqref{eq:plant-bounds-paper} is intentional.
The current day-ahead planner models each plant as continuously dispatchable between zero and its exogenous generation or pumping envelope.
It supplies continuous reference schedules $p_g^{\mathrm{g,DA}}(h)$ and $p_g^{\mathrm{pump,DA}}(h)$ but no plant start-up, shut-down, minimum-output, minimum-up-time, or minimum-down-time decision.Consequently, the converter-plant binary $\delta_i^{\mathrm{DA}}(h)$ cannot be applied to plant $g$.
The modeling extension covering more general scenarios are left as our future work.
\end{remark}

The plant stage cost is
\begin{equation}
\ell_{g}^{\mathrm{pl}}(t,m)=\DeltaID\left[
c_g^{\mathrm g}(t)\pgen_g(t,m)+c_g^{\mathrm p}(t)\ppump_g(t,m)
+\frac{a_g(t)}{2}\bigl(p_g^{\mathrm{pl}}(t,m)\bigr)^2
\right],
\label{eq:plant-cost-paper}
\end{equation}
where $c_g^{\mathrm g}(t),c_g^{\mathrm p}(t),a_g(t)\ge0$ are fixed at issue time.
The generation and pumping cost coefficients $c_g^{\mathrm g}(t)$ and
$c_g^{\mathrm p}(t)$ have units \euro/MWh.
The parameter $a_g(t)$ is the quadratic cost coefficient for the net electrical plant power $p_g^{\mathrm{pl}}$.
Likewise, it has units \euro/$(\mathrm{MW}^2\,\mathrm h)$.
For a hydroelectric unit, $c_g^{\mathrm g}$ is an externally supplied marginal water opportunity value and $c_g^{\mathrm p}$ represents incremental pumping/usage cost.

\begin{proposition}[No simultaneous plant generation and pumping]
\label{prop:no-sim-plant}
If $c_g^{\mathrm g}(t)+c_g^{\mathrm p}(t)>0$, every optimum satisfies
$\pgen_g(t,m)\ppump_g(t,m)=0$.
\end{proposition}

\begin{proof}
If both variables are positive, reduce both by
$\delta=\min\{\pgen_g,\ppump_g\}>0$.  Their difference $p_g^{\mathrm{pl}}$, the nodal
balance, all bounds, and the quadratic term in \eqref{eq:plant-cost-paper} remain unchanged.
The affine plant cost decreases by
$\DeltaID\delta(c_g^{\mathrm g}+c_g^{\mathrm p})>0$, a contradiction.
\end{proof}

\subsection{Renewable dispatch, peak target, and market-energy coupling}

Direct renewable injection satisfies
\begin{equation}
0\le\pres_i(t,m)\le\Pres_i(t,m),\qquad i\in\Rset.
\label{eq:renewable-bounds-paper}
\end{equation}
For a converter area, introduce one horizon peak-exceedance variable per scenario,
\begin{equation}
\speak_i(m)\ge q_i^+(t,m)-\bar p_i^{\mathrm{DA}},
\qquad \speak_i(m)\ge0,\quad\forall t.
\label{eq:peak-slack-paper}
\end{equation}
Using one variable for the horizon maximum avoids charging the same demand-peak event once
per stage.
A stagewise exceedance penalty can be used instead, but it must then be
interpreted as a soft operating penalty rather than as a literal capacity charge.

For a gate-closed hour $h$, let $e_i^{\mathrm{DA}}(h)$ be the cleared converter energy and
$e_{i,h}^{\mathrm{past}}(n)$ the energy already delivered in that hour before control
instant $\tau_n$.  If the current horizon contains every remaining quarter-hour of $h$,
define $G_{n,h}=\{t\in\Tset:h_n(t)=h\}$ and impose, for every scenario,
\begin{equation}
\sum_{t\in G_{n,h}}\pconv_i(t,m)\DeltaID
=e_i^{\mathrm{DA}}(h)-e_{i,h}^{\mathrm{past}}(n).
\label{eq:remaining-market-energy}
\end{equation}
When the day-ahead interface supplies the hourly mean converter power
$b_i^{\mathrm{DA}}(h)$, the corresponding energy is
$e_i^{\mathrm{DA}}(h)=\DeltaDA b_i^{\mathrm{DA}}(h)$.
This remaining-energy equality is causal and does not erase energy already delivered.  An
open hour has no hard equality.  If the horizon does not cover the complete remainder of a
closed hour, exact enforcement requires an additional terminal energy state.
Otherwise, the equality must be deferred or explicitly prorated.

The day-ahead battery set-points enter softly.
Let $E_i^{\mathrm{ref}}(t+1)$ be the absolute-time interpolation of $E_i^{\mathrm{DA}}$ to the end of intraday stage $t$ and $w_i^{\mathrm{soc}}\ge0$.
Define the battery storage tracking cost
\begin{equation}
J_i^{\mathrm{tr}}=
\frac{w_i^{\mathrm{soc}}}{2}
\sum_{m\in\Mset}\pi_m\sum_{t\in\Tset}
\bigl(E_i(t+1,m)-E_i^{\mathrm{ref}}(t+1)\bigr)^2.
\label{eq:soc-tracking-paper}
\end{equation}
Converter power, battery power, renewable dispatch, and accepted braking power remain available for intraday corrections, subject to the physical and market constraints above.

\subsection{First-stage non-anticipativity}

Let the applied control block of area $i$ be
\begin{equation}
u_i(t,m)=\operatorname{col}\bigl(\pconv_i,\pc_i,\pd_i,\pres_i,\preg_i,
(\pgen_g,\ppump_g)_{g\in\Gset_i}\bigr)(t,m),
\end{equation}
with absent-asset entries omitted.
By non-anticipativity, the current control decision cannot depend on a future scenario.
Thus we impose the following non-anticipativity constraint
\begin{equation}
u_i(0,m)=u_i(0,m'),\qquad\forall i,\,\forall m,m'\in\Mset.
\label{eq:nonanticipativity-paper}
\end{equation}
In that way, all scenario exogenous data at $t=0$ are set equal to the current measurement.
Next components of the optimal control sequence may depend on scenario $m$.
This leads to a two-stage scenario fan.

\section{Formal Formulation of Intraday Problem}
\label{sec:intraday-final}

This section assigns economic meaning to the feasible dispatch developed above and assembles the complete finite-scenario optimization problem in a risk-neutral setting.

Recalling the asset models in Section~\ref{sec:intraday-model} and the plant stage cost given by \eqref{eq:plant-cost-paper}, for area $i$, scenario $m$, and stage $t$, define the native operating cost as
\begin{align}
\ell_i(t,m)=\;&
\DeltaID \bigl( \,\underbrace{c_i^E(t)q_i^+(t,m)}_{\text{energy import}}-\underbrace{c_i^{E,\mathrm x}(t)q_i^-(t,m)}_{\text{export credit}}\,\bigr)
\nonumber\\
&+\DeltaID \underbrace{c_i^{\mathrm b}\bigl(\pc_i(t,m)+\pd_i(t,m)\bigr)}_{\text{BESS throughput / aging proxy}}
\nonumber\\
&+\DeltaID \underbrace{c_i^{\mathrm{ct}}\bigl(\Pres_i(t,m)-\pres_i(t,m)\bigr)}_{\text{RES curtailment}} + \underbrace{c^{\mathrm{pk}}\speak_i(m)}_{\text{peak charge}}
\nonumber\\
&+\DeltaID \underbrace{c_i^{\mathrm{rg}}\bigl(\Pav_i(t,m)-\preg_i(t,m)\bigr)}_{\text{un-recovered regen}}
+\underbrace{\sum_{g\in\Gset_i}\ell_g^{\mathrm{pl}}(t,m)}_{\text{railway-plant generation}},
\label{eq:stage-cost-paper}
\end{align}
where penalty coefficients $c_i^{\mathrm b}, c_i^{\mathrm{ct}}, c^{\mathrm{pk}}, c_i^{\mathrm{rg}}, c^{\mathrm{pk}}, c_i^{\mathrm{rg}} \geq 0$.
Import prices $c_i^E(t)$ and export credits $c_i^{E,\mathrm x}(t)$ may be signed.
This changes only affine coefficients and not convexity.
Terms containing only exogenous availability are retained in \eqref{eq:stage-cost-paper} for physical interpretation but may be removed without changing the minimizer.

\begin{remark}
Import prices $c_i^E(t)$ and export credits $c_i^{E,\mathrm x}(t)$ may be positive, zero, or negative.
Since $q_i^+,q_i^-\ge0$, a negative import price $c_i^E(t)<0$ represents a payment received for importing energy, whereas a negative export credit $c_i^{E,\mathrm x}(t)<0$, which makes $-c_i^{E,\mathrm x}(t)q_i^-(t,m)$ positive, represents a payment made for exporting energy.
Negative wholesale electricity prices can arise when high generation coincides with low demand and generation, storage, or demand cannot adjust sufficiently.
The coefficients in this model can reflect such prices when the applicable trading or supply contract passes them through.  They need not be negative under every tariff.
\end{remark}

By decision variables of area $i\in \Nset$ for all $H$ stages and $M$ scenarios, we introduce the following stacked decision vector
\begin{equation}
\begin{aligned}
x_i: =\operatorname{col}\big(&
\pconv_i,q_i^+,q_i^-,\speak_i;\
\pc_i,\pd_i,E_i;\
\pres_i;\
(\pgen_g,\ppump_g)_{g\in\Gset_i};\
\preg_i;\\
&\theta_{i,i};\ (\pflow_{e,i})_{e\ni i};\
(\theta_{j,i})_{j\in\Nbr{i}};\ (\ell_e)_{i=\mathrm{from}(e)}
\big)\in\R^{n_i}.
\end{aligned}
\label{eq:local-vector-paper}
\end{equation}
where $n_i$ is obtained by adding the numbers of scalar coordinates in its blocks, while blocks for absent converter, battery, or renewable assets are omitted:
Each scalar power, angle, flow, or loss variable has $HM$ entries, indexed by $(t,m)\in\Tset\times\Mset$.
The peak block $\speak_i$ has $M$ entries, one per scenario, whereas the battery energy block $E_i$ has $(H+1)M$ entries, indexed by $(s,m)\in\TsetE\times\Mset$.
Let $\mathbf{1}_{\{i\in\mathcal S\}}$ equal one if $i\in\mathcal S$ and zero otherwise, and let $L_i^{\mathrm{out}}\defeq|\{e\in\Eset:i=\mathrm{from}(e)\}|$ be the number of lines whose loss epigraphs are assigned to area $i$.
Since there are $K_i$ incident lines and $K_i$ neighbor-angle copies, the exact dimension is
\begin{equation}
\begin{aligned}
n_i={}&
\underbrace{(3HM+M)\mathbf{1}_{\{i\in\Iset\}}}_{\pconv_i,\,q_i^+,\,q_i^-,\,\speak_i}
+\underbrace{\bigl(2HM+(H+1)M\bigr)\mathbf{1}_{\{i\in\Bset\}}}_{\pc_i,\,\pd_i,\,E_i}
\\
&+\underbrace{HM\mathbf{1}_{\{i\in\Rset\}}}_{\pres_i}
+\underbrace{2HM|\Gset_i|}_{(\pgen_g,\,\ppump_g)_{g\in\Gset_i}}
+\underbrace{HM}_{\preg_i}
\\
&+\underbrace{HM}_{\theta_{i,i}}
+\underbrace{HMK_i}_{(\pflow_{e,i})_{e\ni i}}
+\underbrace{HMK_i}_{(\theta_{j,i})_{j\in\Nbr{i}}}
+\underbrace{HML_i^{\mathrm{out}}}_{(\ell_e)_{i=\mathrm{from}(e)}}.
\end{aligned}
\label{eq:local-vector-dimension-expanded-paper}
\end{equation}

\begin{definition}\label{def:feasible_ser_Xi}
Let $\Xset_i \subset \R^{n_i}$ be the aggregated set defined by the local constraints \eqref{eq:regen-bounds-paper}, \eqref{eq:nodal-balance-paper}--\eqref{eq:nonanticipativity-paper} together with the applicable bounds, initial conditions, day-ahead quantities, and line-loss cuts.
We refer to $\Xset_i$ as \emph{local feasible set of intraday planning}.
\end{definition}

As seen later in Proposition~\ref{prop:local-polyhedron}, the feasible set $\Xset_i$ is a polyhedron.

For $a\in\{i\}\cup\Nbr{i}$, let $S_i^{(a)}\in\{0,1\}^{HM\times n_i}$ select the block $\theta_{a,i}$, we have
\begin{equation}
S_i^{(a)}x_i=\theta_{a,i}\in\R^{HM}.
\label{eq:selector-paper}
\end{equation}
Each row contains one unit entry, and the selected angle blocks are disjoint.  Therefore
$S_i^{(a)}S_i^{(a)\top}=I_{HM}$.

Assign each line-loss epigraph~\eqref{eq:loss-cuts-paper} once, to its from-area.
Considering the battery storage tracking cost~\eqref{eq:soc-tracking-paper} and the line-loss epigraph, the \emph{local objective} is
\begin{align}
f_i(x_i)=\;&
\sum_{m\in\Mset}\pi_m\left[
\sum_{t\in\Tset}\Bigg(
\ell_i(t,m)+
\sum_{e:\,i=\mathrm{from}(e)}c^{\mathrm{ls}}\ell_e(t,m)\DeltaID
\Bigg)
\right] +J_i^{\mathrm{tr}} +\frac{\varepsilon_i}{2}\|x_i\|^2,
\label{eq:local-objective-paper}
\end{align}
where $c^{\mathrm{ls}}>0$ and $\varepsilon_i\ge0$.
The last term on the right-hand side~\eqref{eq:local-objective-paper} is just a numerical regularization term.


Now we formally introduce the finite-scenario problem which will be solved in the next section using consensus ADMM in a distributed way.

\begin{problem}\label{prob:intraday_planning}
Given the network topology $\mathcal G=(\Nset,\Eset)$.
Recall the operating cost~\eqref{eq:stage-cost-paper} and local feasible set of intraday planning $\Xset_i$.
Let $z_a\in\R^{HM}$ be so-called one consensus angle trajectory for any area $a \in \Nset$ and choose a reference area $r$.
Given $x_i \in \Xset_i$ with $x_i$ as in~\eqref{eq:local-vector-paper} and~\eqref{eq:local-vector-dimension-expanded-paper}, the intraday planning is respresented by
\begin{subequations}
\label{eq:ID-paper}
\begin{align}
\min_{\{x_i\},\{z_a\}}\quad&\sum_{i\in\Nset}f_i(x_i),
\label{eq:ID-objective-paper}\\
\mathrm{s.t.}\quad&x_i\in\Xset_i,
&&i\in\Nset,
\label{eq:ID-local-paper}\\
&S_i^{(a)}x_i=z_a,
&&i\in\Nset,\quad a\in\{i\}\cup\Nbr{i},
\label{eq:ID-consensus-paper}\\
&z_r=0.
\label{eq:ID-reference-paper}
\end{align}
\end{subequations}
\end{problem}

\section{Centralized Equivalence, Convexity, and Sparsity}
\label{sec:properties}

This section investigates the structural properties of Problem~\ref{prob:intraday_planning}.
This enables us to use the distributed convex optimization algorithms to solve Problem~\ref{prob:intraday_planning}.
It proves that each area has a polyhedral feasible set, that agreement of duplicated boundary angles exactly recovers the centralized network, that the resulting problem is a convex quadratic program, and that its local matrices retain the stagewise sparsity required for area-wise computation.

\subsection{Polyhedral local sets}

\begin{proposition}
\label{prop:local-polyhedron}
For fixed scenario data, day-ahead inputs, initial conditions, and asset parameters, each $\Xset_i$ as in Definition~\ref{def:feasible_ser_Xi} is a closed convex polyhedron.
\end{proposition}

\begin{proof}
The nodal balance \eqref{eq:nodal-balance-paper}, local flow relation
\eqref{eq:local-flow-paper}, battery dynamics \eqref{eq:bess-dynamics-paper}, converter
split \eqref{eq:converter-split-paper}, plant definition \eqref{eq:plant-net-paper},
remaining-energy condition \eqref{eq:remaining-market-energy}, and non-anticipativity
condition \eqref{eq:nonanticipativity-paper} are affine equalities.  Line limits,
battery and converter bounds, plant envelopes, renewable and regenerative bounds, peak
epigraph inequalities, terminal-energy inequalities, and loss cuts are affine
inequalities.  The fixed day-ahead commitment multiplies constants and therefore introduces
no binary variable into the intraday problem.  Hence $\Xset_i$ is the intersection of
finitely many affine hyperplanes and closed half-spaces.  Such an intersection is a closed
convex polyhedron, possibly empty.
\end{proof}

\subsection{Angle-only consensus and the centralized network}

Let the centralized physical problem use one angle $\theta_a(t,m)$ per area and one flow
$p_e(t,m)=B_e(\theta_{\mathrm{from}(e)}-\theta_{\mathrm{to}(e)})$ per transmission line $e$.  Its
objective is obtained from $\sum_i f_i$ after substituting every local copy by the
corresponding physical angle and assigning every edge cost once, exactly as in
\eqref{eq:local-objective-paper}.

\begin{theorem}[Angle consensus recovers the coupled network]
\label{thm:angle-equivalence}
Let $(\{x_i\},\{z_a\})$ satisfy \eqref{eq:ID-local-paper}--\eqref{eq:ID-reference-paper}.
For every edge $e$ with $i=\mathrm{from}(e)$ and $j=\mathrm{to}(e)$ and every $(t,m)$, we have
\begin{align}
\pflow_{e,i}(t,m)&=B_e\bigl(z_i(t,m)-z_j(t,m)\bigr),
\label{eq:true-flow-from}\\
\pflow_{e,j}(t,m)&=\pflow_{e,i}(t,m).
\label{eq:flow-agreement}
\end{align}
\end{theorem}

\begin{proof}
Fix $e$, $t$, and $m$ and suppress $(t,m)$.  Since $j\in\Nbr{i}$, area $i$ holds both
$\theta_{i,i}$ and $\theta_{j,i}$.  Consensus gives
$\theta_{i,i}=z_i$ and $\theta_{j,i}=z_j$.  Because $i$ is the stored from-node,
$\sigma_{i,e}=+1$, and \eqref{eq:local-flow-paper} gives
\[
\pflow_{e,i}=B_e(\theta_{i,i}-\theta_{j,i})=B_e(z_i-z_j),
\]
which proves \eqref{eq:true-flow-from}.

Area $j$ holds $\theta_{j,j}=z_j$ and $\theta_{i,j}=z_i$.  Since
$\sigma_{j,e}=-1$, its local flow equation gives
\[
\pflow_{e,j}=-B_e(\theta_{j,j}-\theta_{i,j})
=-B_e(z_j-z_i)=B_e(z_i-z_j)=\pflow_{e,i}.
\]
Thus the two node copies agree as a consequence of angle consensus. 
\end{proof}

\begin{corollary}
\label{cor:central-equivalence}
The feasible points of \eqref{eq:ID-paper} are in one-to-one correspondence with feasible points of the centralized physical problem, after the redundant local copies are included or eliminated.
Corresponding points have identical objective values.
Hence, the optimal values are equal, and every optimum of either formulation maps to an optimum of the other.
\end{corollary}

\begin{proof}
Given a feasible point of \eqref{eq:ID-paper}, set the centralized angle to
$\theta_a=z_a$ and the centralized line flow to the common value established by
\cref{thm:angle-equivalence}.  Every nodal balance, line limit, and local asset constraint is
then exactly one of the already satisfied local constraints, and $z_r=0$ fixes the centralized
reference angle.  Since each edge cost is assigned once, substituting the consensus values
leaves the objective unchanged.

Conversely, start from a feasible centralized point.
In every holder $i$ of angle $a$, set $\theta_{a,i}=\theta_a$ and set $z_a=\theta_a$.  Give both nodes of each edge the same centralized flow, and copy every area-local asset variable into its corresponding $x_i$.
All local equations and bounds are then satisfied, and every consensus equality holds by
construction.  This lifting is unique once the centralized variables and the declared block
ordering are fixed.  The two maps are inverses after copy elimination and preserve the
objective, proving the result.
\end{proof}

\subsection{Convex quadratic-program structure}

\begin{assumption}[Standing numerical and economic conditions]
\label{ass:standing-paper}
For each fixed issue time $\tau_n$, the following conditions hold.
\begin{enumerate}[label=(\roman*),leftmargin=2em]
\item The graph $\mathcal G=(\Nset,\Eset)$ is finite and connected,
$H,M,K_\ell\in\N$, and
$B_e,\bar P_e\in(0,\infty)$ for every $e\in\Eset$, as in
\eqref{eq:local-flow-paper}--\eqref{eq:line-limit-paper}.
\item The battery and converter parameters in
\eqref{eq:bess-dynamics-paper}--\eqref{eq:converter-ramp-paper} satisfy
\begin{align*}
&0\le\underline E_i\le\overline E_i<\infty,\quad
\overline P_i^{\mathrm b}\in[0,\infty),\quad
\eta_i^{\mathrm c},\eta_i^{\mathrm d}\in(0,1],\quad
E_i^0,E_i^{\mathrm{term}}\in\R,
&&i\in\Bset,\\
&-\infty<\underline P_i^{\mathrm{cv}}\le0\le
\overline P_i^{\mathrm{cv}}<\infty,\quad
\Delta\overline P_i^{\mathrm{cv}}\in[0,\infty),\quad
p_{i,n}^{\mathrm{conv,prev}}\in\R,
&&i\in\Iset.
\end{align*}
\item The data on the right-hand sides of
\eqref{eq:regen-bounds-paper}, \eqref{eq:nodal-balance-paper},
\eqref{eq:plant-bounds-paper}, and \eqref{eq:renewable-bounds-paper} satisfy
\begin{align*}
&\Pmot_i(t,m),\Pav_i(t,m),\Pres_i(t,m)\in[0,\infty),
&&(i,t,m)\in\Nset\times\Tset\times\Mset,\\
&\overline P_g^{\mathrm g}(t,m),\overline P_g^{\mathrm p}(t,m)\in[0,\infty),
&&(g,t,m)\in\Gset\times\Tset\times\Mset.
\end{align*}
The nonnegative intervals here exclude $+\infty$.  The bounds are imposed only for
applicable assets.  An absent renewable channel is zero.
\item For every applicable index, $\delta_i^{\mathrm{DA}}(h)\in\{0,1\}$,
$\bar p_i^{\mathrm{DA}}\in[0,\infty)$, and
$b_i^{\mathrm{DA}}(h)$, $E_i^{\mathrm{DA}}(h)$,
$p_g^{\mathrm{g,DA}}(h)$, $p_g^{\mathrm{pump,DA}}(h)$,
$e_i^{\mathrm{DA}}(h)$, $e_{i,h}^{\mathrm{past}}(n)$, and
$E_i^{\mathrm{ref}}(t+1)$ are real numbers.  All values $h_n(t)$ used in the problem
belong to $\mathcal H^{\mathrm{DA}}$.  The loss-cut coefficients
$a_{e,\kappa},b_{e,\kappa}$ in \eqref{eq:loss-cuts-paper} are real numbers.
\item The scenario probabilities satisfy $\pi_m\in(0,1]$ and
$\sum_{m\in\Mset}\pi_m=1$.  For all applicable areas, units, and stages,
\begin{gather*}
c_g^{\mathrm g}(t),c_g^{\mathrm p}(t),a_g(t),
c_i^{\mathrm b},c_i^{\mathrm{ct}},c_i^{\mathrm{rg}},
w_i^{\mathrm{soc}},c^{\mathrm{pk}},\varepsilon_i\in[0,\infty),
\qquad c^{\mathrm{ls}}\in(0,\infty),\\
c_i^E(t),c_i^{E,\mathrm x}(t)\in\R,\qquad
c_i^E(t)>c_i^{E,\mathrm x}(t)\quad(i\in\Iset,t\in\Tset).
\end{gather*}
The scenario values and all parameters listed above are fixed during the solution of
the issue-time optimization problem.
\end{enumerate}
\end{assumption}

\begin{theorem}[Convex-QP representation]
\label{thm:convex-qp-paper}
Under \cref{ass:standing-paper}, Problem~\ref{prob:intraday_planning} is a convex quadratic program.
In particular, stacking $v=\operatorname{col}((x_i)_{i\in\Nset},(z_a)_{a\in\Nset}) \in \R^{\sum_{i\in\Nset}n_i+NHM}$, the overall cost function of the intraday planning is represented as the following quadratic program (QP)
\begin{equation}
\min_v\ \frac12v\Tr Qv+q\Tr v
\quad\mathrm{s.t.}\quad A_{\mathrm{eq}}v=b_{\mathrm{eq}},\quad
A_{\mathrm{in}}v\le b_{\mathrm{in}},
\label{eq:standard-qp-paper}
\end{equation}
where $Q=Q\Tr\succeq0$ with $Q =\operatorname{blkdiag}\big(Q_1,\ldots,Q_N,\mathbf0_{NHM\times NHM}\big)$ and $q = \operatorname{col}\big(q_1,\dots,q_N,\mathbf0_{NHM}\big)$, with an explicit representation of matrices $A_{\mathrm{eq}},A_{\mathrm{in}}$, $Q_i$, and $q_i$ as in~\eqref{eq:qp-constraint-blocks-paper}, \eqref{eq:local-hessian-explicit-paper}, and \eqref{eq:local-linear-explicit-paper}, respectively, in Appendix~\ref{app:qp-blocks-paper}.
More particularly, if $\varepsilon_i>0$ for every area, each local physical Hessian block is positive definite.
\end{theorem}

\begin{proof}
Appendix~\ref{app:qp-blocks-paper} constructs the local constraint rows directly
from the model equations.  Its assembly \eqref{eq:qp-constraint-blocks-paper}
adds exactly the consensus and reference-angle constraints.
Expanding the existing local objective gives
\eqref{eq:local-hessian-explicit-paper}--\eqref{eq:local-linear-explicit-paper},
up to a constant independent of the decisions.
The expression for $Q_i$ is $\varepsilon_iI_{n_i}$ plus nonnegative weighted
outer products.  Hence $Q_i\succeq0$, with $Q_i\succ0$ if $\varepsilon_i>0$.
The consensus coordinates add a zero Hessian block, so $Q\succeq0$.
\end{proof}

\begin{remark}
Theorem~\ref{thm:convex-qp-paper} classifies the optimization problem.
It does not assert that $\Xset_i$ or the coupled feasible set is nonempty for every state, day-ahead plan, scenario fan, or horizon.
In this work, we do not address recursive feasibility of the receding-horizon controller.
Such a result would require, for example, an appropriate terminal invariant set, a proved reserve construction, or a formally designed hierarchy of constraint relaxations.
\end{remark}

\subsection{Local-QP sparsity}
\label{sec:local-sparsity-paper}

Throughout this subsection, $x_i$ is the complete horizon-and-scenario vector
in \eqref{eq:local-vector-paper}, with dimension $n_i$ given by
\eqref{eq:local-vector-dimension-expanded-paper}.
Let $d_i=1+K_i$, $g_i=|\Gset_i|$, and retain the previously defined
$L_i^{\mathrm{out}}\le K_i$.
Using $M\le HM$ and that each asset indicator in
\eqref{eq:local-vector-dimension-expanded-paper} is at most one gives directly
\begin{equation}
2HM\le n_i\le HM(11+2g_i+3K_i).
\label{eq:local-dimension-paper}
\end{equation}
For matrix assembly, the first-stage condition
\eqref{eq:nonanticipativity-paper} is represented equivalently by
\begin{equation}
u_i(0,m)-u_i(0,1)=0,\qquad i\in\Nset,\quad m=2,\ldots,M.
\label{eq:nonanticipativity-reference-paper}
\end{equation}
These $M-1$ vector equalities imply every pairwise equality in
\eqref{eq:nonanticipativity-paper} and avoid redundant comparisons.

For a real matrix $A$, define
$\operatorname{nnz}(A)\defeq|\{(j,k):A_{jk}\ne0\}|$.  Let
$A_i=\operatorname{col}(A_{i,\mathrm{eq}},A_{i,\mathrm{in}})\in\R^{r_i\times n_i}$
be the matrix of local equalities and inequalities defining $\Xset_i$, and let
$Q_i=\nabla^2 f_i$ be its objective Hessian.  Each two-sided inequality is represented
by its two one-sided inequalities, and first-stage non-anticipativity is represented by
\eqref{eq:nonanticipativity-reference-paper}.  Battery states and line flows are
retained as variables, so no elimination of their equations is implicit in the following
matrix bounds.

\begin{proposition}[Sparsity of the centralized and local QPs]
\label{prop:sparsity-paper}
For the matrix representation just specified and fixed $K_\ell$, there is a constant
$C>0$, independent of the network size, horizon length, scenario count, and area index,
such that
\begin{equation}
r_i+\operatorname{nnz}(A_i)+\operatorname{nnz}(Q_i)
\le C HM(1+K_i+g_i),\qquad i\in\Nset.
\label{eq:local-nnz-bound-paper}
\end{equation}
For $A=\operatorname{col}(A_{\mathrm{eq}},A_{\mathrm{in}})$ and $Q$ in
\eqref{eq:standard-qp-paper},
\begin{equation}
\operatorname{nnz}(A)+\operatorname{nnz}(Q)
=\mathcal O\!\left(HM\left[N+|\Eset|+|\Gset|+
\sum_{i\in\Nset}K_i\right]\right).
\label{eq:nnz-bound-paper}
\end{equation}
The same order holds for the centralized physical formulation after identifying redundant
angle and node copies of the line flows, while retaining battery states and line flows.
If $K_i$ and $g_i$ are bounded by constants independent of $N,H,M$, then
$r_i,n_i=\Theta(HM)$ and the densities of both $A_i$ and $Q_i$ are
$\mathcal O((HM)^{-1})$.  The stacked constraint matrix has
$\Theta(NHM)$ rows and columns, and its density and the stacked Hessian density are
$\mathcal O((NHM)^{-1})$.
\end{proposition}

\begin{proof}
Following the matrix rows in Appendix~\ref{app:qp-blocks-paper}, a balance row has at most $6+2g_i+K_i$ nonzero entries.
Each flow equation has three, each scalar bound has one, and each loss-cut row has at most two.
Battery dynamics have four entries per row, whereas converter split, ramp, and peak rows have at most three.
Their numbers of rows are bounded by a constant times $HM(1+K_i+g_i+K_\ell L_i^{\mathrm{out}})$, including the $(H+1)M$ energy coordinates and the $M$ initial-energy and peak entries.
For market-energy rows, the sets $G_{n,h}$ are disjoint, so $\sum_h|G_{n,h}|\le H$ gives at most $HM$ nonzeros per converter area.
The non-anticipativity block contributes $2(M-1)\dim u_i(0,1)$ nonzeros, with $\dim u_i(0,1)\le5+2g_i$.
Thus, for fixed $K_\ell$,
\[
r_i\le\operatorname{nnz}(A_i)\le C_1HM(1+K_i+g_i).
\]
The first inequality holds because every retained row has a nonzero coefficient.
By \eqref{eq:local-hessian-explicit-paper}, all non-plant Hessian terms are
diagonal, and each plant contributes at most four entries per stage and scenario.
Consequently,
$\operatorname{nnz}(Q_i)\le n_i+4g_iHM$.
Combining this with \eqref{eq:local-dimension-paper} proves
\eqref{eq:local-nnz-bound-paper}.  The diagonal ADMM penalty in
\eqref{eq:sigma-paper} adds no more than $d_iHM$ entries.

In \eqref{eq:qp-constraint-blocks-paper}, every consensus row has two nonzeros
and each reference row has one.  Hence
\[
\operatorname{nnz}(A)=\sum_i\operatorname{nnz}(A_i)+2HM\sum_i d_i+HM,
\qquad
\operatorname{nnz}(Q)=\sum_i\operatorname{nnz}(Q_i).
\]
The identities $\sum_i g_i=|\Gset|$ and $\sum_iK_i=2|\Eset|$ give
\eqref{eq:nnz-bound-paper}.  Identifying redundant angle and flow copies in the
physical formulation merges columns and removes consensus rows without increasing
the number of nonzeros.  Diagonal and plant Hessian blocks retain the same structure.

If $K_i,g_i$ are uniformly bounded, \eqref{eq:local-dimension-paper} yields
$n_i=\Theta(HM)$.  The $2HM$ regenerative-bound rows give
$r_i\ge2HM$, and the preceding upper bound gives $r_i=\Theta(HM)$.
Dividing the local nonzero bounds by $r_in_i$ and $n_i^2$ proves the local densities.
The stacked formulation has $\sum_i n_i+NHM=\Theta(NHM)$ columns and
$\sum_i r_i+HM\sum_i d_i+HM=\Theta(NHM)$ rows.
Its nonzero bound is $\mathcal O(NHM)$, which gives the stated global densities.
The same orders hold for the physical formulation, which retains area dispatch
coordinates and regenerative bounds.
\end{proof}

This result concerns storage of the original sparse matrices.
It does not bound fill-in during matrix factorization or establish a solver run-time bound.
Those properties also depend on elimination ordering, conditioning, and the optimization algorithm.

\section{Consensus ADMM}
\label{sec:admm-paper}

This section derives the distributed coordination algorithm directly from the finite-scenario problem.
It states the augmented Lagrangian and local quadratic programs, derives the consensus-angle and multiplier updates, defines numerical residuals, and proves centralized optimality of the limiting dispatch under the classical exact-solve assumptions.

\subsection{Augmented Lagrangian and local QPs}

Attach a multiplier $y_i^{(a)}\in\R^{HM}$ to each angle-copy equality.  For fixed
$\rho>0$, the augmented Lagrangian is
\begin{equation}
\begin{aligned}
\mathcal L_\rho(x,z,y)=\sum_{i\in\Nset}\Bigg[f_i(x_i)
+\sum_{a\in\{i\}\cup\Nbr{i}}\Big(&y_i^{(a)\top}(S_i^{(a)}x_i-z_a)\\
&+\frac{\rho}{2}\|S_i^{(a)}x_i-z_a\|^2\Big)\Bigg].
\end{aligned}
\label{eq:augmented-lagrangian-paper}
\end{equation}
At iteration $k$, all areas solve in parallel
\begin{equation}
x_i^{k+1}=\arg\min_{x_i\in\Xset_i}
f_i(x_i)+\sum_{a\in\{i\}\cup\Nbr{i}}
\left[y_i^{(a),k\top}(S_i^{(a)}x_i-z_a^k)
+\frac{\rho}{2}\|S_i^{(a)}x_i-z_a^k\|^2\right].
\label{eq:admm-x-paper}
\end{equation}
Write $f_i(x_i)=\tfrac12x_i\Tr Q_ix_i+q_i\Tr x_i$ up to a constant and define
\begin{align}
\Sigma_i&=\sum_{a\in\{i\}\cup\Nbr{i}}S_i^{(a)\top}S_i^{(a)},
\label{eq:sigma-paper}\\
w_i^k&=\sum_{a\in\{i\}\cup\Nbr{i}}S_i^{(a)\top}
(y_i^{(a),k}-\rho z_a^k).
\label{eq:w-paper}
\end{align}
Since the selected blocks are disjoint, $\Sigma_i$ is diagonal and equal to one exactly on
the angle-copy coordinates.
Therefore,~\eqref{eq:admm-x-paper} cen be written as the local following QP
\begin{equation}
x_i^{k+1}=\arg\min_{x_i\in\Xset_i}
\frac12x_i\Tr(Q_i+\rho\Sigma_i)x_i+(q_i+w_i^k)\Tr x_i.
\label{eq:local-admm-qp-paper}
\end{equation}

\subsection{Consensus and multiplier updates}

Let $\mathcal H_a=\{a\}\cup\Nbr{a}$ be the holders of angle $a$ and $d_a=|\mathcal H_a|=1+K_a$.
For $a\ne r$, collect the $z_a$-dependent terms of \eqref{eq:augmented-lagrangian-paper} and differentiate $0=\sum_{i\in\mathcal H_a}\left[-y_i^{(a),k}+\rho(z_a-S_i^{(a)}x_i^{k+1})\right]$.
Solving the resulting linear equation for $z_a$, which is equivalent to minimizing \eqref{eq:augmented-lagrangian-paper} over $z_a$, yields
\begin{equation}
z_a^{k+1}=\frac1{d_a}\sum_{i\in\mathcal H_a}
\left(S_i^{(a)}x_i^{k+1}+\frac1\rho y_i^{(a),k}\right),
\qquad z_r^{k+1}=0.
\label{eq:admm-z-paper}
\end{equation}
For the reference area, the feasible set contains only $z_r=0$, so its constrained minimizer is exactly zero.
The multiplier update is
\begin{equation}
y_i^{(a),k+1}=y_i^{(a),k}
+\rho\bigl(S_i^{(a)}x_i^{k+1}-z_a^{k+1}\bigr).
\label{eq:admm-y-paper}
\end{equation}
Equation~\eqref{eq:admm-y-paper} is gradient ascent on the unscaled multiplier with step $\rho$ for the corresponding equality constraint.
Thus one coordination step exchanges only boundary-angle trajectories between adjacent areas.
Converter setpoints, battery states, costs, and internal forecasts need not be sent to other areas.

\subsection{Residuals and stopping quantities}

Stack all copy mismatches to obtain the \emph{primal} residual
\begin{equation}
r^{k+1}=\operatorname{col}_{i,a}
\bigl(S_i^{(a)}x_i^{k+1}-z_a^{k+1}\bigr).
\label{eq:primal-residual-paper}
\end{equation}
Since each $z_a$ is repeated $d_a$ times in the constraint stack, the corresponding \emph{dual} residual norm is
\begin{equation}
\|s^{k+1}\|=
\rho\left(
\sum_{a\in\Nset}d_a\|z_a^{k+1}-z_a^k\|^2
\right)^{1/2}.
\label{eq:dual-residual-paper}
\end{equation}
Following the standard absolute/relative criteria~\cite{boyd2011}, the ADMM algorithm stops when $\|r^{k+1}\|_2\le\epsilon^{\mathrm{pri}}$ and $\|s^{k+1}\|_2\le\epsilon^{\mathrm{dual}}$, where
\begin{align}
\epsilon^{\mathrm{pri}} &= \sqrt{p}\,\epsilon^{\mathrm{abs}}
+ \epsilon^{\mathrm{rel}}\max\Bigl\{\bigl\|(S_i^{(a)}x_i)_{i,a}\bigr\|_2,\ \bigl\|(z_a)_{i,a}\bigr\|_2\Bigr\},\\
\epsilon^{\mathrm{dual}} &= \sqrt{p}\,\epsilon^{\mathrm{abs}}
+ \epsilon^{\mathrm{rel}}\bigl\|(y_i^{(a)})_{i,a}\bigr\|_2 .
\end{align}

To impose a prescribed accuracy on every angle copy, we supplement the above criteria by the infinity norm of the primal residual
\begin{equation}
c_\theta^{k+1}
=\max_{\substack{i\in\Nset\\a\in\{i\}\cup\Nbr{i}}}
\bigl\|S_i^{(a)}x_i^{k+1}-z_a^{k+1}\bigr\|_\infty
=\|r^{k+1}\|_\infty
<\epsilon_\infty^\theta,
\label{eq:componentwise-consensus-paper}
\end{equation}
with some fixed tolerance $\epsilon_\infty^\theta>0$.
Infinity-norm primal-residual tests are standard in operator-splitting QP solvers
\cite[Sec.~3.4]{stellato2020}.
The Euclidean test only guarantees each error is at most $\epsilon^{\mathrm{pri}}$, whose value depends on problem size and iterate scale.
Equation~\eqref{eq:componentwise-consensus-paper} instead bounds every scalar angle-copy error, at every stage and scenario, by the same prescribed tolerance.
Its physical meaning follows directly from \eqref{eq:local-flow-paper}.
For $e=(i,j)$ in its stored orientation, exact local flow equations and the triangle inequality give
\begin{equation}
\begin{aligned}
\|p_{e,i}^{\mathrm{flow},k+1}-B_e(z_i^{k+1}-z_j^{k+1})\|_\infty
&\le2B_ec_\theta^{k+1},\\
\|p_{e,i}^{\mathrm{flow},k+1}-p_{e,j}^{\mathrm{flow},k+1}\|_\infty
&\le4B_ec_\theta^{k+1}.
\end{aligned}
\label{eq:consensus-flow-error-paper}
\end{equation}
Thus the test controls the disagreement of the two computed line-flow trajectories.

\subsection{Convergence and centralized optimality}

\begin{assumption}
\label{ass:admm-paper}
Suppose that Problem~\eqref{eq:ID-paper} has a finite optimal value and its unaugmented Lagrangian has a
saddle point.  Each local update \eqref{eq:admm-x-paper} is solved exactly, the consensus
update is \eqref{eq:admm-z-paper}, and $\rho>0$ is fixed.
\end{assumption}

\begin{theorem}[ADMM convergence]
\label{thm:admm-convergence-paper}
Under \cref{ass:standing-paper,ass:admm-paper}, the iterations
\eqref{eq:admm-x-paper}, \eqref{eq:admm-z-paper}, and \eqref{eq:admm-y-paper} satisfy
\begin{equation}
r^k\to0,
\qquad
\sum_{i\in\Nset}f_i(x_i^k)\to J^\star,
\qquad
y^k\to y^\star,
\label{eq:admm-convergence-result}
\end{equation}
where $J^\star$ is the optimal value of \eqref{eq:ID-paper} and $y^\star$ is a dual
solution.  Every accumulation point of the primal iterates is optimal.  Through
\cref{cor:central-equivalence}, the limiting physical dispatch is also an optimum of the
centralized problem.
\end{theorem}

\begin{proof}
Define the closed proper convex functions $F(x)=\sum_{i\in\Nset}\bigl(f_i(x_i)+\iota_{\Xset_i}(x_i)\bigr)$, $G(z)=\iota_{\{z:z_r=0\}}(z)$, where $\iota_C$ is zero on $C$ and $+\infty$ outside $C$.
Closedness and convexity follow from \cref{prop:local-polyhedron,thm:convex-qp-paper}
Properness follows from the assumed existence of a feasible optimum.
Stack the selectors into $A$ and let $B$ repeat each $-z_a$ in the rows associated with its holders.
Then \eqref{eq:ID-paper} is exactly the two-block problem $\min_{x,z}\ F(x)+G(z)$ subject to $Ax+Bz=0$.
The $x$-minimization of its augmented Lagrangian separates by area and is precisely
\eqref{eq:admm-x-paper}.
The constrained $z$-minimization is \eqref{eq:admm-z-paper} and the multiplier ascent is \eqref{eq:admm-y-paper}.
Hence these iterations are the standard two-block ADMM iterations.
The saddle-point assumption is the regularity hypothesis of the classical ADMM convergence theorem~\cite[Sec.~3.2]{boyd2011}.
Applying that theorem gives primal-residual convergence, objective convergence, and dual convergence in \eqref{eq:admm-convergence-result}.
It also implies optimality of every primal accumulation point.
Finally,~\cref{cor:central-equivalence} maps any such point to a centralized feasible point with the same globally minimal objective.
\end{proof}

\begin{remark}
Theorem~\ref{thm:admm-convergence-paper} concerns exact local minimizers.
Established inexact-ADMM results allow local errors whose norms are summable over the outer iterations, but a fixed loose error without such control is not covered by \cref{thm:admm-convergence-paper}.
Numerical tolerances, scaling, stopping, and recovery are therefore Part~II implementation matters~\cite{norooziP1II} and are tested in Part~III~\cite{norooziP1III}.
\end{remark}

While the companion paper Part~II~\cite{norooziP1II} is devoted to numerical implementation of Problem~\ref{prob:intraday_planning} via the consensus ADMM as in Section~\ref{sec:admm-paper}, we present the workflow of the numerical algorithm in pseducodes, solving  the  
\begin{algorithm}[H]
\caption{Consensus ADMM for the intraday dispatch problem (ID)}
\label{alg:admm}
\begin{algorithmic}[1]
\Require Network $(\Nset,\Eset)$; local data $(Q_i,q_i,\Xset_i)$; penalty $\rho$;
tolerances $\epsilon^{\mathrm{abs}},\epsilon^{\mathrm{rel}}, \epsilon_\infty^\theta$; maximum iteration $K_{\max}$
\State \textbf{init} $x_i^{0}$, $z_a^{0}$, $y_i^{(a),0}$
\State build and factorise the local programs for $Q_i + \rho^0\Sigma_i$,
\For{$k=0,1,2,\dots,K_{\max}-1$}
  \ForAll{$i\in\Nset$ \textbf{in parallel}}
    \State $w_i^{k}\gets$ compute~\eqref{eq:w-paper}
    \Comment{only the linear term changes}
    \State $x_i^{k+1}\gets$ solve local QP~\eqref{eq:local-admm-qp-paper}
  \EndFor
  \ForAll{$a\in\Nset$} \Comment{coordination: one exchange with neighbors}
    \State $z_a^{k+1}\gets$ compute~\eqref{eq:admm-z-paper}
  \EndFor
  \State $z_r^{k+1}\gets 0$ \Comment{reference bus}
  \ForAll{$i\in\Nset,\ a\in\{i\}\cup\Nbr{i}$ \textbf{in parallel}}
    \State $y_i^{(a),k+1}\gets$ compute~\eqref{eq:admm-y-paper}
  \EndFor
  \State compute $\|r^{k+1}\|,\|s^{k+1}\|$ and $c_{\theta}^{k+1}$, resp., via~\eqref{eq:primal-residual-paper},~\eqref{eq:dual-residual-paper},~\eqref{eq:componentwise-consensus-paper}
  \If{$\|r^{k+1}\|\le\epsilon^{\mathrm{pri}}$ \textbf{and}
       $\|s^{k+1}\|\le\epsilon^{\mathrm{dual}}$ \textbf{and}
       $c_{\theta}^{k+1}<\epsilon_\infty^\theta$}
    \State \textbf{return} $\{x_i^{k+1}\},\{z_a^{k+1}\}$ \Comment{converged}
  \EndIf
\EndFor
\State \textbf{return} $\{x_i^{K_{\max}}\},\{z_a^{K_{\max}}\}$ \Comment{maximum iteration reached}
\end{algorithmic}
\end{algorithm}

\section{Receding-Horizon Application and KPI Accumulation}
\label{sec:kpi}

This section closes the loop between the finite-horizon optimizer and physical operation.
It defines which common first-stage action is applied, how the measured battery state advances
to the next control instant, and how economic and physical key performance indices are
accumulated without repeatedly counting prediction-horizon quantities.

\subsection{Applied first stage and state propagation}

Let $(x_{i,n}^\star,z_{a,n}^\star)$ solve \eqref{eq:ID-paper} at control instant
$\tau_n$.  Only the first-stage control, common to all scenarios by
\eqref{eq:nonanticipativity-paper}, is applied:
\begin{equation}
u_{i,n}=u_{i,n}^\star(0,m),\qquad m\in\Mset.
\label{eq:applied-first-stage-paper}
\end{equation}
The physical battery state propagates as
\begin{equation}
E_{i,n+1}=E_{i,n}+\DeltaID\left(
\eta_i^{\mathrm c}\pc_{i,n}
-\frac{1}{\eta_i^{\mathrm d}}\pd_{i,n}
\right).
\label{eq:closed-loop-soc-paper}
\end{equation}
The next optimization is initialized with this measured state.

\subsection{Realized economic and physical indices}

Prediction-horizon costs must not be summed repeatedly as realized performance.  Realized
KPIs are accumulated only from the applied first stages and realized prices.  Over
closed-loop indices $n=0,\ldots,N_{\mathrm{cl}}-1$, corresponding to control instants
$\tau_n$, define
\begin{align}
C^{\mathrm{grid}}
&=\sum_n\sum_{i\in\Iset}\DeltaID
\left(c_{i,n}^E q_{i,n}^+
-c_{i,n}^{E,\mathrm x}q_{i,n}^-\right),
\label{eq:kpi-grid-cost}\\
C^{\mathrm{plant}}
&=\sum_n\sum_{g\in\Gset}\DeltaID\left[
c_{g,n}^{\mathrm g}p_{g,n}^{\mathrm g}
+c_{g,n}^{\mathrm p}p_{g,n}^{\mathrm{pump}}
+\frac{a_{g,n}}2(p_{g,n}^{\mathrm g}-p_{g,n}^{\mathrm{pump}})^2
\right],
\label{eq:kpi-plant-cost}\\
E^{\mathrm{imp}}&=\sum_{n,i\in\Iset}q_{i,n}^+\DeltaID,
\qquad
E^{\mathrm{exp}}=\sum_{n,i\in\Iset}q_{i,n}^-\DeltaID,
\label{eq:kpi-grid-energy}\\
E^{\mathrm{BESS,thr}}
&=\sum_{n,i\in\Bset}(p_{i,n}^{\mathrm c}+p_{i,n}^{\mathrm d})\DeltaID,
\label{eq:kpi-bess-throughput}\\
E^{\mathrm{curt}}
&=\sum_{n,i\in\Rset}(P_{i,n}^{\mathrm{res,max}}-p_{i,n}^{\mathrm{res}})\DeltaID,
\label{eq:kpi-curtailment}\\
E^{\mathrm{reg,av}}
&=\sum_{n,i}\Pav_{i,n}\DeltaID,
\quad
E^{\mathrm{reg,rec}}=\sum_{n,i}p_{i,n}^{\mathrm{reg}}\DeltaID,
\quad
E^{\mathrm{reg,spill}}=\sum_{n,i}(\Pav_{i,n}-p_{i,n}^{\mathrm{reg}})\DeltaID,
\label{eq:kpi-regen-energy}\\
\etareg_{\mathrm{net}}
&=\frac{E^{\mathrm{reg,rec}}}{E^{\mathrm{reg,av}}},
\qquad E^{\mathrm{reg,av}}>0,
\label{eq:kpi-regen-ratio}\\
P_{i,n}^{\mathrm{peak}}
&=\max_{0\le j\le n}q_{i,j}^+,
\qquad
\chi_e^{\mathrm{line}}=\max_n\frac{|p_{e,n}|}{\bar P_e}.
\label{eq:kpi-peak-line}
\end{align}
Area-wise recovery ratios are defined analogously whenever the area denominator is positive.
The realized dispatch cost is $C^{\mathrm{grid}}+C^{\mathrm{plant}}$. 

\begin{proposition} \label{prop:kpi-identities}
For every feasible applied sequence, we have
\begin{equation}
E^{\mathrm{reg,av}}=E^{\mathrm{reg,rec}}+E^{\mathrm{reg,spill}},
\qquad
0\le\etareg_{\mathrm{net}}\le1
\label{eq:regen-kpi-identity}
\end{equation}
when $E^{\mathrm{reg,av}}>0$.  Moreover,
$P_{i,n+1}^{\mathrm{peak}}\ge P_{i,n}^{\mathrm{peak}}$ for every $i$ and $n$.
\end{proposition}

\begin{proof}
By~\eqref{eq:regen-bounds-paper}, we have $0\le p_{i,n}^{\mathrm{reg}}\le P_{i,n}^{\mathrm{av}}$.
Subtracting accepted from available power, multiplying by $\DeltaID$, and summing over all areas and control instants gives the equality in \eqref{eq:regen-kpi-identity}.
Moreover, nonnegativity of all three energies follows from the same bounds.
Dividing $0\le E^{\mathrm{reg,rec}}\le E^{\mathrm{reg,av}}$ by the strictly positive denominator gives the ratio bounds.
Finally, the set over which the maximum defining $P_{i,n}^{\mathrm{peak}}$ is taken is contained in the corresponding set at $n+1$.
a maximum over a superset cannot be smaller.
\end{proof}

\begin{theorem}
\label{thm:closed-loop-causality}
Suppose that the current state, active day-ahead interface, scenario trajectories, and all exogenous coefficients used at instant $\tau_n$ are $\mathcal F_{\tau_n}$-measurable, stage zero equals the current measurement.
Then $u_n$ in \eqref{eq:applied-first-stage-paper} is $\mathcal F_{\tau_n}$-measurable.
If \eqref{eq:closed-loop-soc-paper} is used for state propagation, the statement holds inductively at every later control instant for which the optimization remains feasible.
\end{theorem}

\begin{proof}
At instant $\tau_n$, all matrices, vectors, and bounds of \eqref{eq:ID-paper} are functions of $\mathcal F_{\tau_n}$-measurable quantities.
Therefore, $u_n$ is $\mathcal F_{\tau_n}$-measurable.
The state update~\eqref{eq:closed-loop-soc-paper} is a deterministic function of the current $\mathcal F_{\tau_n}$-measurable state and control, so $E_{n+1}$ is available at the next instant.
Repeating the same argument proves the result by induction.
The induction cannot continue through an infeasible optimization, which is why feasibility is stated explicitly.
\end{proof}

\section{Causal Timetable-Driven 15-Minute Synthesis}
\label{sec:synthesis}

This section constructs a causal, timetable-driven intra-hour shape while preserving every declared hourly energy exactly.
It first derives train-level electrical power, then assigns that power to control areas and quarter-hour bins, proves conservation and causality, and finally explains when the synthesis is executed and how one identified base trajectory is shared by the two energy management layers.

\subsection{Train-run and longitudinal model}

For a timetabled train run $\varrho$, let
$\mathcal S_\varrho=(s_1,\ldots,s_{n_\varrho})$ be the stop sequence,
$\mathcal Q_\varrho=((\tau_{\varrho,k}^{\mathrm{arr}},
\tau_{\varrho,k}^{\mathrm{dep}}))_{k=1}^{n_\varrho}$ the issued absolute UTC arrival and
departure times, $\mathsf c_\varrho$ the traffic category, and $m_\varrho^{\mathrm{tr}}$ the train
mass.  On the leg from $s_k$ to $s_{k+1}$, the booked running time and route length are
$T_{\varrho,k}$ and $d_{\varrho,k}$.

The longitudinal train dynamics use the classical Davis resistance model
\cite{davis1926,rochard2000}.  In standard railway notation,
\begin{equation}
\lambda_\varrho^{\mathrm{rot}}m_\varrho^{\mathrm{tr}}\dot v_\varrho(\vartheta)
=F_\varrho^{\mathrm{trac}}(\vartheta)
-\bigl(A_\varrho+B_\varrho v_\varrho(\vartheta)
+C_\varrho v_\varrho(\vartheta)^2\bigr)
-m_\varrho^{\mathrm{tr}}g\sin\gamma,
\label{eq:train-dynamics}
\end{equation}
where $\vartheta\in\TUTC$ is absolute time,
$\lambda_\varrho^{\mathrm{rot}}>0$ is the rotating-mass factor, and $\gamma$ is the line gradient
\cite{steimel2014}.  A trapezoidal speed profile accelerates at $a^+>0$, cruises at
$v^{\mathrm{cr}}$, and brakes at $a^-<0$.  Define
\begin{equation}
\kappa_v\defeq\frac12\left(\frac{1}{a^+}+\frac{1}{|a^-|}\right).
\end{equation}
The distance closure is
\begin{equation}
\kappa_v(v^{\mathrm{cr}})^2-T_{\varrho,k}v^{\mathrm{cr}}+d_{\varrho,k}=0.
\label{eq:cruise-quadratic}
\end{equation}

\begin{proposition}[Admissible cruise speed]
\label{prop:cruise-speed}
If $T_{\varrho,k}^2\ge4\kappa_vd_{\varrho,k}$, the unique root of
\eqref{eq:cruise-quadratic} that yields a nonnegative cruise duration is
\begin{equation}
v^{\mathrm{cr}}=
\frac{T_{\varrho,k}-\sqrt{T_{\varrho,k}^2-4\kappa_vd_{\varrho,k}}}{2\kappa_v}.
\label{eq:cruise-speed}
\end{equation}
If $T_{\varrho,k}^2<4\kappa_vd_{\varrho,k}$, the booked time is infeasible for the assumed
accelerations.  A triangular profile may still preserve the distance by taking
$v^{\mathrm{pk}}=\sqrt{d_{\varrho,k}/\kappa_v}$, but its running time must then be relaxed.
\end{proposition}

\begin{proof}
The two roots of \eqref{eq:cruise-quadratic} are
$v_\pm=(T\pm\sqrt{T^2-4\kappa_vd})/(2\kappa_v)$.  The total acceleration and braking time
is
\[
T^{\mathrm{acc}}+T^{\mathrm{br}}=
v^{\mathrm{cr}}\left(\frac{1}{a^+}+\frac{1}{|a^-|}\right)=2\kappa_vv^{\mathrm{cr}}.
\]
For $v_+$ this value is $T+\sqrt{T^2-4\kappa_vd}>T$ unless the discriminant is zero, so the
implied cruise duration is negative.  For $v_-$ it is
$T-\sqrt{T^2-4\kappa_vd}\le T$, so the cruise duration is nonnegative.  Hence only
$v_-$ is physically admissible.  When the discriminant is negative, no trapezoidal profile
with the booked time exists.  Setting the cruise duration to zero gives
$d=\kappa_v(v^{\mathrm{pk}})^2$, which proves the triangular formula and also shows that its
duration $2\kappa_vv^{\mathrm{pk}}=2\sqrt{\kappa_vd}$ exceeds the infeasible booked time.
\end{proof}

The tractive force follows from \eqref{eq:train-dynamics}, and the wheel power is
$P_\varrho^{\mathrm{mech}}(\vartheta)
=F_\varrho^{\mathrm{trac}}(\vartheta)v_\varrho(\vartheta)$.  With drivetrain efficiency
$\eta_{\mathrm d}$, regenerative fraction $\eta_{\mathrm{rb}}$, and auxiliary power
$P_\varrho^{\mathrm{aux}}$, define the pantograph power
\begin{equation}
p_\varrho^{\mathrm{el}}(\vartheta)=
\begin{cases}
P_\varrho^{\mathrm{mech}}(\vartheta)/\eta_{\mathrm d}+P_\varrho^{\mathrm{aux}},
&P_\varrho^{\mathrm{mech}}(\vartheta)\ge0,\\
\eta_{\mathrm d}\eta_{\mathrm{rb}}P_\varrho^{\mathrm{mech}}(\vartheta)+P_\varrho^{\mathrm{aux}},
&P_\varrho^{\mathrm{mech}}(\vartheta)<0.
\end{cases}
\label{eq:pantograph-power}
\end{equation}
The gross motoring and available regenerative contributions of
this train are therefore
\begin{equation}
\pi_\varrho^{\mathrm{mot}}(\vartheta)= \max\{p_\varrho^{\mathrm{el}}(\vartheta),0\},
\qquad
\pi_\varrho^{\mathrm{av}}(\vartheta)=\max\{-p_\varrho^{\mathrm{el}}(\vartheta),0\}.
\label{eq:train-split}
\end{equation}
Consequently, \eqref{eq:train-split} separates positive electrical demand from
the magnitude of negative electrical power without allowing either channel to be negative.

\subsection{Area assignment, quarter-hour templates, and concentration}

Let $\chi(x)$ assign each track position to the electrical control area that supplies it.
For $h\in\mathcal H^{\mathrm{DA}}$ and $k\in\{0,\ldots,\nu-1\}$, $\nu=4$, define the
quarter-hour bin
\begin{equation}
\mathcal B_{h,k}\defeq
\left[\tau_0^{\mathrm{DA}}+h\DeltaDA+k\DeltaID,\,
\tau_0^{\mathrm{DA}}+h\DeltaDA+(k+1)\DeltaID\right).
\label{eq:quarter-hour-bin}
\end{equation}
For this hourly grid, $\DeltaDA=\nu\DeltaID$.  The raw
timetable shapes are obtained with the proposition indicator
\begin{equation}
\mathbf 1\{\mathsf P\}\defeq
\begin{cases}
1,&\text{if the proposition $\mathsf P$ is true},\\
0,&\text{if the proposition $\mathsf P$ is false}.
\end{cases}
\label{eq:indicator-definition}
\end{equation}
Thus $\mathbf 1\{\chi(x_\varrho(\vartheta))=i\}$ assigns the instantaneous contribution of train
$\varrho$ to area $i$ and to no other area.  The raw powers are
\begin{align}
\widetilde P_{i,h,k}^{\mathrm{mot}}
&=\frac{1}{\DeltaID}\sum_\varrho\int_{\mathcal B_{h,k}}
\mathbf 1\{\chi(x_\varrho(\vartheta))=i\}\pi_\varrho^{\mathrm{mot}}(\vartheta)\,\mathrm d\vartheta,
\label{eq:raw-motor}\\
\widetilde P_{i,h,k}^{\mathrm{av}}
&=\frac{1}{\DeltaID}\sum_\varrho\int_{\mathcal B_{h,k}}
\mathbf 1\{\chi(x_\varrho(\vartheta))=i\}\pi_\varrho^{\mathrm{av}}(\vartheta)\,\mathrm d\vartheta.
\label{eq:raw-regen}
\end{align}
In a public-data-constrained specialization, complete train paths may be unavailable.  The
same equations are then evaluated with confirmed category counts, representative train-leg
profiles, and clock-face arrival and departure templates.  That specialization remains a
model-derived shape, not a measurement.

For $q\in\{\mathrm{mot},\mathrm{av}\}$, normalize the nonnegative clock-face template as
\begin{equation}
\widehat w_{i,h,k}^{q}=
\begin{cases}
\widetilde P_{i,h,k}^{q}/\sum_{j=0}^{\nu-1}\widetilde P_{i,h,j}^{q},
&\sum_j\widetilde P_{i,h,j}^{q}>0,\\
1/\nu,&\text{otherwise}.
\end{cases}
\label{eq:normalized-template}
\end{equation}
The concentration parameter $c\in[0,1]$ blends this clock-face shape with a uniform shape,
\begin{equation}
w_{i,h,k}^{q}(c)=\frac{1-c}{\nu}+c\widehat w_{i,h,k}^{q},
\qquad
P_{i,h,k}^{q}(c)=\nu\,\overline P_i^{q}(h)w_{i,h,k}^{q}(c),
\label{eq:concentration}
\end{equation}
where $\overline P_i^{q}(h)$ is the hourly mean target.
Thus $c=0$ gives four equal quarter-hour values, $c=1$ gives the fully clock-face-shaped template, and an intermediate value retains part of the timetable structure while smoothing exact synchronization.
The parameter is a benchmark modelling choice.

\begin{remark}[Physical interpretation of the concentration parameter]
The two limiting shapes in \eqref{eq:concentration} have distinct meanings.  The uniform
shape $1/\nu$ spreads each hourly channel equally over all quarter-hours.  The clock-face
shape $\widehat w^q$ places braking near the assumed arrival bins, acceleration near the
corresponding departure bins, and the more persistent cruising contribution across the
intervening bins.  The scalar $c$ controls the strength of this synchronization.  In
particular, $c=1$ selects the most concentrated shape permitted by the declared clock-face
template.
It does not assert perfect train punctuality.
A value such as $c=0.85$ retains 85 percent of the template departure from uniformity and 15 percent uniform smoothing.
\end{remark}

\subsection{Conservation, causality, and execution contract}

\begin{theorem}[Exact hourly conservation and nonnegativity]
\label{thm:synthesis-conservation}
For every $c\in[0,1]$, area $i$, hour $h$, and channel
$q\in\{\mathrm{mot},\mathrm{av}\}$, the synthesis \eqref{eq:concentration} satisfies
\begin{equation}
P_{i,h,k}^{q}(c)\ge0,
\qquad
\frac1\nu\sum_{k=0}^{\nu-1}P_{i,h,k}^{q}(c)=\overline P_i^{q}(h).
\label{eq:hourly-conservation}
\end{equation}
Consequently it changes only the intra-hour allocation and neither creates nor removes
hourly energy.
\end{theorem}

\begin{proof}
By \eqref{eq:normalized-template}, $\widehat w_{i,h,k}^{q}\ge0$ and
$\sum_k\widehat w_{i,h,k}^{q}=1$, including the zero-raw-energy fallback.  Since
$c\in[0,1]$, \eqref{eq:concentration} is a convex combination of two nonnegative unit-sum
vectors.  Hence $w_{i,h,k}^{q}(c)\ge0$ and $\sum_kw_{i,h,k}^{q}(c)=1$.  Multiplication by
$\nu\overline P_i^q(h)\ge0$ proves nonnegativity, and
\[
\frac1\nu\sum_kP_{i,h,k}^{q}(c)
=\overline P_i^q(h)\sum_kw_{i,h,k}^{q}(c)=\overline P_i^q(h),
\]
which proves exact conservation.  Multiplying both sides by one hour gives equality of the
hourly energies.
\end{proof}

When multiplicative rescaling is unsuitable, the closest nonnegative profile to a raw
template may instead be obtained from
\begin{align}
\min_{P_{i,h,k}^{q}\ge0}\quad &
\sum_{k=0}^{\nu-1}a_{i,h,k}^{q}
(P_{i,h,k}^{q}-\widetilde P_{i,h,k}^{q})^2
+\varsigma\sum_{k=0}^{\nu-2}(P_{i,h,k+1}^{q}-P_{i,h,k}^{q})^2
\label{eq:projection-qp}\\
\text{s.t.}\quad&\frac1\nu\sum_{k=0}^{\nu-1}P_{i,h,k}^{q}
=\overline P_i^{q}(h),
\nonumber
\end{align}
with $a_{i,h,k}^{q}>0$ and $\varsigma\ge0$.  Its Hessian is the sum of a positive diagonal
matrix and a positive-semidefinite first-difference Gramian.
Hence,~\eqref{eq:projection-qp} is a strictly convex QP on a nonempty polyhedron and has one unique solution.
This least-distortion alternative preserves \eqref{eq:hourly-conservation} exactly.

\begin{theorem}[Causality of the synthesis]
\label{thm:synthesis-causal}
Fix an issue time $\tau_n$.
If the timetable or confirmed train counts, route and rolling-stock parameters, and hourly targets used in \eqref{eq:train-dynamics}--\eqref{eq:projection-qp} are all $\mathcal F_{\tau_n}$-measurable, and no realized measurement later than $\tau_n$ is used, then the synthesized quarter-hour profiles are $\mathcal F_{\tau_n}$-measurable.
\end{theorem}

\begin{proof}
For fixed inputs, integration of the deterministic train equations, area assignment,
binning, normalization, convex mixing, and the unique minimizer of
\eqref{eq:projection-qp} are deterministic maps.  A deterministic Borel-measurable map of
$\mathcal F_{\tau_n}$-measurable inputs is $\mathcal F_{\tau_n}$-measurable.  Since the
construction contains no realized value after $\tau_n$, its output cannot depend on future
measurements.
\end{proof}

\ifieeetran\appendices\else\appendix\fi
\section{Computation of QP matrices}
\label{app:qp-blocks-paper}

We use the block order of \eqref{eq:local-vector-paper}, with stages increasing
inside each scenario and fixed orders for edges and units.
For a scalar coordinate $y$ of $x_i$, define the selector row
$R_i[y]\in\{0,1\}^{1\times n_i}$ by placing one in the column of $y$ and zeros
elsewhere.  Thus $R_i[y]x_i=y$.
In this appendix only, arguments $(t,m)$ on powers, flows, and angles inside
selectors are suppressed when the row or sum is indexed by $(t,m)$.
Selectors for absent assets are zero in sums, and their individual constraints
and cost terms are omitted.  Empty blocks contribute no rows or columns.

For every applicable index, stack the following coefficient rows and right-hand
sides to obtain $A_{i,\mathrm{eq}}$ and $b_{i,\mathrm{eq}}$.
Here $U_i(t,m)$ stacks the selectors of the coordinates of the already defined
$u_i(t,m)$, so that $U_i(t,m)x_i=u_i(t,m)$.
\begin{longtable}{@{}p{0.72\linewidth}p{0.24\linewidth}@{}}
\toprule
Equality coefficient row or block & Right-hand side\\
\midrule
\endhead
$\begin{aligned}[t]
&R_i[\pconv_i]+R_i[\pd_i]-R_i[\pc_i]+R_i[\pres_i]+R_i[\preg_i]\\
&\quad+\sum_{g\in\Gset_i}\bigl(R_i[\pgen_g]-R_i[\ppump_g]\bigr)
-\sum_{e\ni i}\sigma_{i,e}R_i[\pflow_{e,i}]
\end{aligned}$
& $\Pmot_i(t,m)$\\[3pt]
$R_i[\pflow_{e,i}]-\sigma_{i,e}B_e(R_i[\theta_{i,i}]-R_i[\theta_{j,i}])$,
$e$ joining $i,j$
& $0$\\[3pt]
$\begin{aligned}[t]
&R_i[E_i(t+1,m)]-R_i[E_i(t,m)]\\
&\quad-\DeltaID\eta_i^{\mathrm c}R_i[\pc_i]
+\DeltaID(\eta_i^{\mathrm d})^{-1}R_i[\pd_i]
\end{aligned}$, $i\in\Bset$
& $0$\\[3pt]
$R_i[E_i(0,m)]$, $i\in\Bset$ & $E_i^0$\\[3pt]
$R_i[\pconv_i]-R_i[q_i^+]+R_i[q_i^-]$, $i\in\Iset$ & $0$\\[3pt]
$\DeltaID\sum_{t\in G_{n,h}}R_i[\pconv_i(t,m)]$,
$i\in\Iset$, gate-closed $h$ satisfying \eqref{eq:remaining-market-energy}
& $e_i^{\mathrm{DA}}(h)-e_{i,h}^{\mathrm{past}}(n)$\\[3pt]
$U_i(0,m)-U_i(0,1)$, $m=2,\ldots,M$ & $0$\\
\bottomrule
\end{longtable}
Rows involving a stage are included for every $(t,m)\in\Tset\times\Mset$.
Initial-energy and non-anticipativity rows occur only at their displayed indices.
Market-energy rows occur once per applicable nonempty $G_{n,h}$ and scenario.

For each coordinate bound $\underline y\le y\le\overline y$ with endpoints
fixed at the issue time in
\eqref{eq:regen-bounds-paper}, \eqref{eq:line-limit-paper},
\eqref{eq:bess-box-paper}, \eqref{eq:converter-split-paper},
\eqref{eq:frozen-commitment-paper}, \eqref{eq:plant-bounds-paper},
\eqref{eq:renewable-bounds-paper}, and \eqref{eq:peak-slack-paper}, include
\begin{equation}
\begin{bmatrix}-R_i[y]\\R_i[y]\end{bmatrix}x_i
\le\begin{bmatrix}-\underline y\\\overline y\end{bmatrix}.
\label{eq:qp-box-rows-paper}
\end{equation}
A missing or infinite endpoint contributes no row.
Here the endpoints are exactly the bounds already specified in those equations.
In particular, the bounds on $E_i(s,m)$ use $s=0,\ldots,H$, whereas the
nonnegativity of $\speak_i(m)$ contributes one row per scenario.
Append the following rows, using $\zeta\in\{-1,1\}$ for both signs, to form
$A_{i,\mathrm{in}}$ and $b_{i,\mathrm{in}}$.
\begin{longtable}{@{}p{0.72\linewidth}p{0.24\linewidth}@{}}
\toprule
Additional inequality coefficient row & Right-hand side\\
\midrule
\endhead
$-R_i[E_i(H,m)]$, $i\in\Bset$ & $-E_i^{\mathrm{term}}$\\[3pt]
$\zeta R_i[\pconv_i(0,m)]$, $i\in\Iset$
& $\Delta\overline P_i^{\mathrm{cv}}+\zeta p_{i,n}^{\mathrm{conv,prev}}$\\[3pt]
$\zeta(R_i[\pconv_i(t,m)]-R_i[\pconv_i(t-1,m)])$,
$i\in\Iset$, $t\ge1$
& $\Delta\overline P_i^{\mathrm{cv}}$\\[3pt]
$R_i[q_i^+]-R_i[\speak_i(m)]$, $i\in\Iset$
& $\bar p_i^{\mathrm{DA}}$\\[3pt]
$a_{e,\kappa}R_i[\pflow_{e,i}]-R_i[\ell_e]$,
$i=\mathrm{from}(e)$, $\kappa=1,\ldots,K_\ell$
& $-b_{e,\kappa}$\\
\bottomrule
\end{longtable}
Each local matrix has $n_i$ columns.  Define
$\widehat S_i=\operatorname{col}_{a\in\{i\}\cup\Nbr{i}}S_i^{(a)}$,
$\widehat S=\operatorname{blkdiag}_i\widehat S_i$, and
$Z_a=\mathbf e_a^\top\otimes I_{HM}$, where $\mathbf e_a$ is the $a$th
standard basis vector of $\R^N$ and $\otimes$ denotes the Kronecker product.
Thus $Z_a$ selects $z_a$ from $z=\operatorname{col}_a z_a$.
Stack $\widehat Z=\operatorname{col}_{i\in\Nset,\,a\in\{i\}\cup\Nbr{i}}Z_a$
in the row order of $\widehat S$.
With $D_{\mathrm{eq}}=\operatorname{blkdiag}_iA_{i,\mathrm{eq}}$,
$D_{\mathrm{in}}=\operatorname{blkdiag}_iA_{i,\mathrm{in}}$, and
$p=HM\sum_i(1+K_i)$, the coupled matrices are
\begin{equation}
\begin{aligned}
A_{\mathrm{eq}}&=\begin{bmatrix}D_{\mathrm{eq}}&\mathbf0\\
\widehat S&-\widehat Z\\
\mathbf0&Z_r\end{bmatrix},
&b_{\mathrm{eq}}&=\operatorname{col}\bigl((b_{i,\mathrm{eq}})_i,\mathbf0_p,\mathbf0_{HM}\bigr),\\
A_{\mathrm{in}}&=\begin{bmatrix}D_{\mathrm{in}}&\mathbf0\end{bmatrix},
&b_{\mathrm{in}}&=\operatorname{col}_i b_{i,\mathrm{in}}.
\end{aligned}
\label{eq:qp-constraint-blocks-paper}
\end{equation}
Here $r$ is the reference area.
Both matrices have $\sum_i n_i+NHM$ columns.
Multiplication by $x_i$ recovers the original local constraints row by row.
The $M-1$ non-anticipativity comparisons imply all pairwise comparisons.
The last two block rows of $A_{\mathrm{eq}}$ impose angle consensus and $z_r=0$.
This proves exact equality of the feasible sets.

Use \eqref{eq:stage-cost-paper} and \eqref{eq:local-objective-paper} directly.
For each $g\in\Gset_i$, define the row
$D_{ig}(t,m)=R_i[\pgen_g(t,m)]-R_i[\ppump_g(t,m)]$.
The local Hessian is
\begin{equation}
\begin{aligned}
Q_i=\varepsilon_iI_{n_i}+\sum_{m\in\Mset}\pi_m\sum_{t\in\Tset}\Big[&
w_i^{\mathrm{soc}}R_i[E_i(t+1,m)]^\top R_i[E_i(t+1,m)]\\
&+\DeltaID\sum_{g\in\Gset_i}a_g(t)D_{ig}(t,m)^\top D_{ig}(t,m)\Big].
\end{aligned}
\label{eq:local-hessian-explicit-paper}
\end{equation}
All selectors in the following sum have the same $(t,m)$ unless another index
is displayed.  The linear coefficient is
\begin{equation}
\begin{aligned}
q_i={}&\sum_{m\in\Mset}\pi_m\sum_{t\in\Tset}\Bigg\{
\DeltaID\Big[
c_i^E(t)R_i[q_i^+]-c_i^{E,\mathrm x}(t)R_i[q_i^-]\\
&\quad+c_i^{\mathrm b}\bigl(R_i[\pc_i]+R_i[\pd_i]\bigr)
-c_i^{\mathrm{ct}}R_i[\pres_i]-c_i^{\mathrm{rg}}R_i[\preg_i]\\
&\quad+\sum_{g\in\Gset_i}\bigl(c_g^{\mathrm g}(t)R_i[\pgen_g]
+c_g^{\mathrm p}(t)R_i[\ppump_g]\bigr)
+c^{\mathrm{ls}}\sum_{e:\,i=\mathrm{from}(e)}R_i[\ell_e]\Big]\\
&\quad-w_i^{\mathrm{soc}}E_i^{\mathrm{ref}}(t+1)R_i[E_i(t+1,m)]
\Bigg\}^{\!\top}
+Hc^{\mathrm{pk}}\sum_{m\in\Mset}\pi_mR_i[\speak_i(m)]^\top.
\end{aligned}
\label{eq:local-linear-explicit-paper}
\end{equation}
Each product in \eqref{eq:local-hessian-explicit-paper} is an $n_i\times n_i$
matrix and each transposed row in \eqref{eq:local-linear-explicit-paper} is
an $n_i$-vector.  Battery terms occur only for $i\in\Bset$, converter and peak
terms only for $i\in\Iset$, and renewable terms only for $i\in\Rset$.

To verify these expressions, expand the tracking squares in
\eqref{eq:soc-tracking-paper} and the net-plant-power squares in
\eqref{eq:plant-cost-paper}.  They give exactly the two sums in $Q_i$ and
the tracking term in $q_i$.  Each remaining coefficient in $q_i$ is its affine
coefficient in \eqref{eq:local-objective-paper}.
The factor $H$ on the peak term follows from its placement inside the stage sum.
Thus, without changing $\ell_i$,
$f_i(x_i)=\tfrac12x_i^\top Q_ix_i+q_i^\top x_i+f_i(0)$.
Every outer product in $Q_i-\varepsilon_iI_{n_i}$ has a nonnegative weight,
so $Q_i\succeq\varepsilon_iI_{n_i}\succeq0$.
Finally, the consensus coordinates do not appear in the cost, giving
$Q=\operatorname{blkdiag}(Q_1,\ldots,Q_N,\mathbf0_{NHM\times NHM})$
and $q=\operatorname{col}(q_1,\ldots,q_N,\mathbf0_{NHM})$.

\begin{singlespace}
\bibliographystyle{IEEEtran}
\bibliography{Bahnstrom_EMS_P1_Common}
\end{singlespace}

\end{document}